\documentclass[a4paper,11pt]{amsart}
\usepackage{xcolor}
\usepackage{amsmath}
\usepackage{amsthm}
\usepackage{amssymb}
\usepackage{comment}
\usepackage{mathrsfs}
\usepackage{graphicx}
\usepackage{subcaption}
\usepackage{float}
\usepackage{enumitem}
\usepackage{calc}
\usepackage{tikz}
\usepackage{tikz-cd}
\usepackage{pgfplots}
\usepackage{adjustbox}
\usetikzlibrary{intersections, calc}
\usetikzlibrary{patterns}
\usepackage{appendix}
\usepackage{subcaption}
\usetikzlibrary{arrows,shapes,positioning}
\usetikzlibrary{decorations.markings}
\tikzstyle arrowstyle=[scale=1]
\tikzstyle directed=[postaction={decorate,decoration={markings,
    mark=at position 0.65 with {\arrow[arrowstyle]{latex}}}}]
\tikzstyle reverse directed=[postaction={decorate,decoration={markings,
    mark=at position 0.5 with {\arrowreversed[arrowstyle]{latex};}}}]
\usetikzlibrary{decorations.pathreplacing}
\theoremstyle{definition}
\newtheorem{thm}{Theorem}[section]
\newtheorem{prop}{Proposition}[section]
\newtheorem{dfn}{Definition}[section]

\newtheorem{lem}{Lemma}[section]
\newtheorem{cor}{Corollary}[section]

\newenvironment{demo}[1]{%
  \trivlist
  \item[\hskip\labelsep
        {\bf #1.}]
}{%
\hfill\qedsymbol
  \endtrivlist
}

\newcommand{\mex}{\operatorname{mex}}
\newcommand{\nimadd}{\overset{*}{+}}
\newcommand{\nimsum}{{\sum}^{*}}
\newcommand{\nimmul}{\overset{*}{\times}}

\numberwithin{equation}{section}
\def\cqs{{(\lower1mm\hbox{*})}}
\newlength{\cqswidth}
\def\Non{\mathbb{N}}
\def\Pos{\mathbb{P}}
\def\Int{\mathbb{Z}}
\def\P{\mathcal{P}}
\def\Par{\mathscr{P}}
\def\F{\mathbb{F}}
\def\Game{\mathscr{A}}
\def\Pgame{\mathscr{G}}

\def\T{\mathscr{T}}
\def\TT{\mathscr{T}\kern-4pt\mathscr{T}}
\def\OI{\mathscr{O}}
\def\I{\mathscr{I}}

\def\len#1{\ell\left({#1}\right)}
\def\ASM#1{\boldsymbol{A}_{#1}}
\def\MOD{\operatorname{mod}}
\def\Y#1{\mathbb{Y}_{#1}}
\def\R#1{\mathcal{R}_{#1}}
\def\e{\boldsymbol{e}}
\def\x{\boldsymbol{x}}
\def\y{\boldsymbol{y}}
\def\z{\boldsymbol{z}}
\def\zero{\boldsymbol{0}}

\def\fnh#1{\mathscr{R}\left({#1}\right)}
\def\fnH#1#2{\Phi_{#1}^{#2}}
\def\fnS#1{\mathscr{S}\left({#1}\right)}
\def\rkf{\rho}
\def\fnb{\mathfrak{p}}

\def\digit#1#2{\gamma_{#1}\left({#2}\right)}
\def\placemax#1{d_{\max}\left({#1}\right)}
\def\placemin#1{d_{\min}\left({#1}\right)}
\def\qbinom#1#2#3{\left[{{#1}\atop{#2}}\right]_{#3}}

\def\bnu{\boldsymbol{\nu}}
\begin{document}
\title[Coin Turning Games on Partially Ordered Sets]
{
	Coin Turning Games \\on Partially Ordered Sets
}
\author{M.~Ishikawa}
\address{Masao Ishikawa, Department of Mathematics, Okayama University}
\email{mi@math.okayama-u.ac.jp}
\author{T.~Ohmoto}
\address{Toyokazu Ohmoto, Department of Mathematics, Okayama University}
\email{toyokazuohmoto@gmail.com}
\author{H. Tagawa}
\address{Hiroyuki Tagawa, Department of Mathematics, Wakayama University}
\email{tagawa@wakayama-u.ac.jp}
\author{Y.~Takayama}
\address{Yoshiki Takayama, Department of Mathematics, Okayama University}
\email{y.t.yoshiki.0611@gmail.com}
%
%
\maketitle
\begin{abstract}
A finite impartial game is a two-player game 
in which the players take turns making moves and the game ends after finitely many moves.
In this paper, we study a class of finite impartial games introduced by H.~Lenstra, 
which we call coin turning games.
We focus on two typical classes of coin turning games, 
namely the order ideal games and the rulers, 
distinguished by their choices of turning sets.
For several posets arising from enumerative combinatorics, 
we determine the Sprague-Grundy functions.
In particular, we determine the Sprague-Grundy function of the order ideal game on the ASM poset $\ASM{n}$, 
introduced by J.~Striker in connection with the alternating sign matrices.
\end{abstract}
\tableofcontents
%
%
%
%
%
\section{Introduction}
A coin turning game is a finite impartial game, that is, a type of combinatorial game.
It is played by two players who alternately make moves, and the game terminates after finitely many moves.
H.~Lenstra \cite{Lenstra} studied coin turning games on an arbitrary poset $X$
together with a turning set $\T$,
which is a family of subsets of $X$ satisfying a certain condition (see \S~\ref{ssc:XTgame}).
In this paper, we investigate poset games on several posets arising in combinatorial studies (see, e.g., \cite{ec1}), 
including the poset $B_n(q)$ of all subspaces of the $n$-dimensional vector space over a finite field $\F_{q}$,
the poset $\Pi_{n}$ of set partitions.
Of particular importance is the poset
\begin{equation}\label{eq:ASM-set}
	\ASM{n}
	=
	\left\{
		(x, y, z) \in \Int^{3}
	\, \middle| \,
		x, y, z \geq 0
		\textrm{ and }
		x + y + z \leq n - 2
	\right\},
\end{equation}
which arises as the set of join-irreducible elements of the distributive lattice of alternating sign matrices of size $n$ (see \cite{Br,Ku,MRR1,MRR2,Oh,St}).
This poset was introduced by J.~Striker \cite{SW,St} in connection with the study of generalizations of toggles and gyrations.
The partial order on $\ASM{n}$ is defined as follows (see \cite[Definition~3.3]{St}):
for $(x_{1}, y_{1}, z_{1}),(x_{2}, y_{2}, z_{2})\in\ASM{n}$,
we say $(x_{1}, y_{1}, z_{1})\leq (x_{2}, y_{2}, z_{2})$ if and only if
\begin{equation}\label{eq:ASM-order}
x_{1}\geq x_{2},\, y_{1}\geq y_{2},\, z_{1}\leq z_{2} \text{ and } x_{1}+y_{1}+z_{1}\geq x_{2}+y_{2}+z_{2}.
\end{equation}
Throughout this paper, we refer to $\ASM{n}$ as the \textsl{the ASM poset}.
\par\smallskip
A poset game \cite{Lenstra,Sato} depends not only on the underlying poset $X$ 
but also on the choice of the turning set $\T$.
Common choices of $\T$ lead to well-known games \cite{Sato}.
If $\T$ consists of all two-element subsets $\{a,b\}$ with $a\leq b$,
the resulting game is called turning turtles.
If $\T$ consists of all closed intervals $[a,b]$ with $a\leq b$,
the game is called a ruler.
In addition to these, we introduce another class of poset games obtained by taking $\T$
to be the set of all principal order ideals of $X$ (see \S~\ref{subsec:poset} for the definition).
We refer to the resulting game as the order ideal game (or simply the ideal game).
\par\smallskip
Our goal is to obtain explicit formulas for the Sprague-Grundy functions \cite{WW,ONAG,Grundy,Sprague}
(see Definition~\ref{def:Grundy}) of poset games.
One of the main objectives of this paper is to determine the Sprague-Grundy function of the order ideal game on $\ASM{n}$.
Let $n$ be a positive integer and set $\T=\{\Lambda_{\x}\mid \x\in \ASM{n}\}$,
where $\Lambda_{\x}=\{\y\in\ASM{n}\mid \y\leq\x\}$ denotes the principal order ideal generated by $\x\in\ASM{n}$.
By Lenstra's theorem (see Theorem~\ref{thm:FT_of_T_game}), 
it suffices to determine the Sprague-Grundy value
$g(x, y, z)=g(\{(x, y, z)\})$ for each $(x, y, z) \in \ASM{n}$.
We shall show that $g(x, y, z)$ depends only on two parameters, 
namely the rank $\rho(x,y,z)=n-2-(x+y)$ and $z$.
\begin{thm}\label{th:ASM-ideal}
	For each 
	$(x, y, z) \in \ASM{n}$,
	the Sprague--Grundy value
	of the order ideal game on $\ASM{n}$
	is given by
	\begin{equation}\label{eq:ASM-ideal-grundy}
		g(x, y, z)
		=
		\begin{cases}
			1	&\qquad	\text{$\rho(x,y,z)=0$ or  $\rho(x,y,z)=2z \pm 1$,}\\
			0	&\qquad	\text{otherwise.}
		\end{cases}
	\end{equation}
\end{thm}
%
%
%
This paper originated as the master's thesis \cite{Ta} of the fourth author.
The paper is organized as follows.
In Section~\ref{sec:ImpartialGames} 
we recall basic facts about finite impartial games and the Sprague-Grundy function.
Readers who are familiar with the elementary theory of finite impartial games may safely skip this section.
In Section~\ref{sec:poset}, 
we introduce notation and preliminaries on partially ordered sets, 
and then describe the definition and basic properties of the ASM poset $\ASM{n}$ in \S~\ref{ssc:posetASM}.
Readers well acquainted with elementary poset theory may skip the first half of this section and return to it as needed.
In Section~\ref{sc:PosetGame},
we define coin turning games and prove the fundamental theorem of coin turning games, namely
Theorem~\ref{thm:FT_of_T_game} (due to Lenstra~\cite{Lenstra}), which is used throughout the paper.
Section~\ref{sc:SG-Analysis} is the main section of the paper. There we compute Sprague-Grundy functions for several types of coin turning games.
Specifically, we study order ideal games and rulers on chains, divisor posets, 
and subspace lattices over finite fields, as well as the order ideal game on the ASM poset.
We conclude by discussing open problems concerning the ruler on the ASM poset and on the poset of set partitions in Section~\ref{sc:open}.
In Appendix~\ref{sc:appendix},
we summarize useful properties of Nim addition and Nim multiplication that are used in the proofs.
In Appendix~\ref{sc:RC},
we present a proof of a characterization of the ruler sequence.
%

%
%
%
%
%
\section{Finite Impartial Games}\label{sec:ImpartialGames}
%
In this section, we define the class of combinatorial games (see \cite{MHRG,WW,ONAG,Sato,Siegel})
that will be studied in the remainder of the paper.
Readers who are familiar with impartial games may safely skip this section and return to it as needed.
%
%
\subsection{Basic notions of finite impartial games}
Any game considered in this paper is played by two players,
and the players alternatively make a move. 
The winner is the player who makes the last move. 
An \textsl{impartial game} is a game satisfying the following conditions (see \cite{MHRG,Siegel}):
\begin{enumerate}[label=(\alph*)] 
		\item
			The possible moves from each position are completely determined by that position
			(that is, there are no \textsl{chance moves} such as rolling dice or shuffling cards).
		\item
			Both players have complete information about the game state at every position.
		\item
			Both players have the same set of available moves from each position.
		\item
			From any given position, only finitely many positions are reachable, and no position can appear twice during a play.
	\end{enumerate} 
	%
%
\begin{dfn}[\cite{WW,ONAG,Lenstra,Sato,Siegel}]\label{def:game}
A game is a pair $(\P, N)$,
where $\P$ is the set of all possible positions
and $N:\P \to 2^{\P}$ is a map that assigns to each position $P\in\P$ 
the set of all possible positions $P'$ that can be reached from $P$ in a single move.
If $P'\in N(P)$,
we say $P'$  is an \textsl{option} of the current position $P$,
and we write $P \rightarrow P'$.
If $N(P)=\emptyset$,
then $P$ is called an \textsl{ending position}.
We denote by $\epsilon$ the set of all ending positions.
The game determined by the pair $(\P,N)$ is denoted by $\Game(\P,N)$.
\end{dfn}
	Starting from a position ${P}_{0}\in\P$,
	the game is played by two players who alternatively make moves,
	\[
	{P}_{0} \rightarrow {P}_{1} \rightarrow {P}_{2} \rightarrow \cdots \rightarrow {P}_{m},
	\]
	where at each position $P_i$, the player in turn chooses an option $P_{i+1}\in N(P_i)$.
If ${P}_{m}\in\epsilon$,
	we call such a sequence a \textsl{play} (or a \textsl{game sequence}) 
	with the \textsl{starting position} $P_{0}$
	and the \textsl{ending position} $P_{m}$,
	and we say that its \textsl{length} is $m$.
Throughout this paper, we assume that every play has finite length.
The player who makes the last move $P_{m-1}\to P_{m}$ is called the \textsl{winner},
and the opposing player is called the \textsl{looser}.
For a position $P\in\P$, let $\len{P}$ denote the maximum length among all plays starting from $P$;
this number is called the \textsl{length} of the position $P$.
It follows immediately that, if $P'$ is an option of $P$, then $\len{P'}<\len{P}$.
\par\smallskip
For any proper subset $T\subsetneq \Non$,
where $\Non$ denotes the set of non-negative integers, we define
$\mex\left( T \right)=\min\left( \Non \setminus T \right)$,
and call $\mex\left( T \right)$ the \textsl{minimal-excluded number} of $T$.
In particular, $\mex( \emptyset ) = 0$.
%
The Sprague-Grundy function is defined inductively using the $\mex$ operator; see \cite{Grundy,Sprague} for details. 
%
\begin{dfn}[Sprague-Grundy function \cite{Grundy,Sprague}]\label{def:Grundy}
Let $\Game = \Game(\P, N)$ be a finite impartial game.
	The \textsl{Sprague-Grundy value} (or \textsl{Grundy value} for short) ${g}_{\Game}(P) \in \Non$ 
	of a position $P \in \P$ is defined inductively by the following conditions:
	\begin{enumerate}[label=(\roman*)] 
		\item\label{it:Grundy1}
			${g}_{\Game}(P) = 0$
			if $P$ is an ending position (equivalently, $\ell(P)=0$);
		\item\label{it:Grundy2}
			$
				{g}_{\Game}(P) 
				=
				\mex\left\{
					{g}_{\Game}(P')
				\, \middle| \,
					P' \in N(P)
				\right\}
			$
			if $\ell(P)>0$.
	\end{enumerate} 
\end{dfn}
This inductive definition determines a function $g_{\Game}:\P\to\Non$ which we call the \textsl{Sprague-Grundy function} (or \textsl{Grundy function} for short) of $\Game$.
For brevity, we will often write $g(P)$ in place of $g_{\Game}(P)$.
%
%
%
\subsection{Combined game}
Given two finite impartial games $\Game_1$ and $\Game_2$,
we give the definition of the combined game $\Game_1+\Game_2$
in which, on each turn, a player can either make a move in the first game, while leaving the second game untouched, 
or make a move in the second game, while leaving the first game untouched.
The precise definition is given below.
%
\begin{dfn}
Let $\Game_1=\Game(\P_1,N_1)$ and $\Game_2=\Game(\P_2,N_2)$ be finite impartial games.
Define the set $\P$ of positions to be the Cartesian product $\P=\P_1\times \P_2$.
For a position $P=(P_1,P_2)\in \P$, we define its set of options by
\begin{equation}\label{eq:CombinedGame}
	N(P)=\{(P_1',P_2),(P_1,P_2')\ |\ P_1'\in N_{1}(P_1),P_2'\in N_{2}(P_2)\}.
\end{equation}
We call the game $\Game=(\P,N)$ the \textsl{combined game} of $\Game_1$ and $\Game_2$,
and denote it by $\Game=\Game_1+\Game_2$.
\end{dfn}
%
%
%
\begin{dfn}[\cite{ONAG,Lenstra,Sato}]\label{def:Nimadd}
If $a$ and $b$ are non-negative integers,
the \textsl{NIM-sum} $a \nimadd b$ of $a$ and $b$ is defined inductively by
\begin{equation}\label{eq:nimadd}
a \nimadd b 
:= 
\mex\bigl(
	\{ a' \nimadd b \mid 0 \leq a' < a \}
	\cup
	\{ a \nimadd b' \mid 0 \leq b' < b \}
\bigr) .
\end{equation}
This inductive definition determines a binary operation
$\nimadd \colon \Non \times \Non \rightarrow \Non$, $(a, b) \mapsto a \nimadd b$,
which is also referred to as \textsl{NIM-addition}.
Let $A$ be a finite subset of $\Non$.
We denote the NIM-sum over the set $A$
by ${\nimsum}_{a \in A}a$.
In particular, the NIM-sum over the empty set is $0$, which follows immediately from the definition \eqref{eq:nimadd}.
\end{dfn}
We summarize the useful properties of NIM-addition in the Appendix~\ref{sc:appendix}.
\par
If $\epsilon_1$ and $\epsilon_2$ denote the set of the ending positions of $\Game_1$ and $\Game_2$, respectively,
then it is immediate that $\epsilon=\epsilon_1\times\epsilon_2$ is the set of ending positions of the combined game $\Game=\Game_1+\Game_2$.
Furthermore, for any position $P=(P_1,P_2)\in \P$,
it follows directly from the definition that $\len{P}=\len{P_1,P_2}=\len{P_1}+\len{P_2}$.
%
%
%
\begin{thm}[\cite{Lenstra}]\label{gamewathm}
Let $\Game_1=\Game(\P_1,N_1)$ and $\Game_2=\Game(\P_2,N_2)$ be finite impartial games
with Grundy functions $g_{\Game_1}$ and $g_{\Game_2}$, respectively.
For the combined game $\Game=\Game_1+\Game_2=\Game(\P,N)$,
the Grundy function is given by
\begin{equation}
	g_{\Game_1+\Game_2}(P_1,P_2)=g_{\Game_1}(P_1)\nimadd g_{\Game_2}(P_2)
\end{equation}
for all $(P_1,P_2)\in \P=\P_1\times\P_2$.
\end{thm}
\begin{demo}{Proof}
We proceed by induction on $\len{P_1,P_2}$.
If $\len{P_1,P_2}=0$, then $(P_1,P_2)\in\epsilon$.
Hence we have $P_1\in \epsilon_1$ and $P_2\in \epsilon_2$ which implies
\[
	g_{\Game}(P_1,P_2)=0=0\nimadd 0=g_{\Game_1}(P_1)\nimadd g_{\Game_2}(P_2).
\]
Now let $\len{P_1,P_2}>0$ and assume that  
\[
	g_{\Game}(\overline{P_1},\overline{P_2})=g_{\Game_1}(\overline{P_1})\nimadd g_{\Game_2}(\overline{P_2})
\]
holds for any pair $(\overline{P_1},\overline{P_2})$ with $\len{\overline{P_1},\overline{P_2}}<\len{P_1,P_2}$.
By the definition \eqref{eq:CombinedGame} of combined game and the Grundy value, we have
\begin{align*}
g_{\Game}(P_1,P_2)
			&=\mex\{g_{\Game}(P_1',P_2),g_{\Game}(P_1,P_2')\ |\ P_1'\in N(P_1),P_2'\in N(P_2)\}.
\end{align*}
Since $\len{P_1',P_2}<\len{P_1,P_2}$ and $\len{P_1,P_2'}<\len{P_1,P_2}$,
the induction hypothesis yields
\begin{align*}
	&g_{\Game}(P_1,P_2)=\mex\{g_{\Game_1}(P'_1)\nimadd g_{\Game_2}(P_2),
	\\&\qquad\qquad g_{\Game_1}(P_1)\nimadd g_{\Game_2}(P_2')\ |\ P_1'\in N(P_1),P_2'\in N(P_2)\}.
\end{align*}
Since
$\mex\{g_{\Game_1}(P'_1)\ |\ P_1'\in N(P_1)\}=g_{\Game_1}(P_1)$ and 
$\mex\{g_{\Game_2}(P'_2)\ |\ P_2'\in N(P_2)\}=g_{\Game_2}(P_2)$
by \eqref{eq:nimadd},
we obtain
\[
	g_{\Game}(P_1,P_2)=g_{\Game_1}(P_1)\nimadd g_{\Game_2}(P_2)
\]
by Lemma~\ref{lemma:NIM_add_and_mex}.
\end{demo}
%
%
%
\subsection{NIM multiplication}
%
%
\begin{dfn}[\cite{ONAG,Lenstra,Sato}]\label{def:NimMul}
If $a$ and $b$ are nonnegative integers,
the \textsl{NIM multiplication} $a \nimmul b$ of $a$ and $b$ is defined inductively by
\begin{equation}\label{eq:NIM-mul}
a \nimmul b 
:= 
\mex\bigl\{a' \nimmul b \nimadd a \nimmul b' \nimadd a' \nimmul b'\mid 0 \leq a' < a \text{ and } 0 \leq b' < b\bigr\}.
\end{equation}
It follows immediately from the definition that $a \nimmul b$ is commutative, 
and that $a \nimmul 0=0$ and $a \nimmul 1=a$ hold for every $a\in\Non$.
Further useful properties of NIM multiplication are summarized in Appendix~\ref{sc:appendix}.
\end{dfn}
%

%
%
%
%
%
\section{Partially Ordered Sets}\label{sec:poset}
In this section, we review notation and preliminaries on partially ordered sets and introduce the specific posets studied in this paper (see \cite{ec1} for details).
Readers who are familiar with poset theory may safely skip the basic definitions and consult \S~\ref{ssc:posetASM}
for the properties of the ASM poset $\ASM{n}$.

%
\subsection{Preliminaries on Posets}\label{subsec:poset}
	A \textsl{partially ordered set} (or \textsl{poset}, for short) $P$ 
	is a set equipped with a binary relation, denoted by $\leq$,
	satisfying the following axioms for all $x, y, z \in P$:
		(P1)
		$x \leq x$,
		(P2)
		$x = y$ if $x \leq y$ and $y \leq x$,
		(P3)
		$x \leq z$ if $x \leq y$ and $y \leq z$.
	The relation $\leq$ is called the \textsl{partial order} on $P$.
	We write $x < y$
	if $x \leq y$ and $x \neq y$.
	We may also denote the poset by $(P, {\leq}_{P})$ (or simply $(P, \leq)$)
	to emphasize the underlying order relation.
	The axioms (P1), (P2) and (P3) are called 
	\textsl{reflectivity}, \textsl{antisymmetry} and \textsl{transitivity}, respectively.
Two elements $x,y \in P$ are said to be \textsl{comparable}
if either $x \leq y$ or $y \leq x$; otherwise they are \textsl{incomparable}.
\par
A \textsl{chain} (or \textsl{totally ordered set}) 
is a poset in which any two elements are comparable.
An \textsl{antichain} 
is a subset $A$ of a poset $P$ such that 
any two distinct elements of $A$ are incomparable.
If $x < y$ and no element $u \in P$ satisfies $x < u < y$,
we say that $y$ \textsl{covers} $x$ 
(or $x$ is \textsl{covered} by $y$), 
and write $y \gtrdot x$ (or $x \lessdot y$). 
\par
We say that $P$ has a minimum element $\hat0$ if there exists an element $\hat0\in P$ such that $t\geq\hat0$ for all $t\in P$.
Similarly, $P$ has a maximum element $\hat1$ if there exists $\hat1\in P$ such that $t\leq\hat1$ for all $t\in P$. 
\par
If $P$ and $Q$ are posets, 
then the \textsl{direct (or cartesian) product} of $P$ and $Q$ is the poset $P\times Q$
on the set $\left\{ (s,t) \,\middle|\, s\in P\text{ and } t \in Q\right\}$ 
with order relation defined by $(s,t) \leq (s',t')$ in $P\times Q$ if and only if $s \leq s'$ in $P$ and $t \leq t'$ in $Q$.
%
\par
Let $(P, {\leq}_{P})$ and $(Q, {\leq}_{Q})$ be posets.
We say that a map $\varphi \colon P \rightarrow Q$ is \textsl{order-preserving}
if $\varphi(s) \, {\leq}_{Q} \, \varphi(t)$
for any $s, t \in P$ such that $s \, {\leq}_{P} \, t$.
Further more, if $\varphi$ is bijective and the inverse map is order-preserving,
then we say that $\varphi$ is \textsl{order isomorphism},
$P$ and $Q$ are \textsl{order isomorphic},
denoted $P \cong Q$.
If $C$ is a chain with $|C| = |P|$, 
then an order preserving bijection $\sigma: P \to C$ is called a \textsl{linear extension} of $P$. 
\par
By an \textsl{induced subposet} of a poset $P$, 
we mean a subset $Q$ of $P$ 
equipped with the partial order inherited from $P$; that is,
for $s,t\in Q$, we have $s \leq t$ in $Q$ if and only if $s \leq t$ in $P$.
Then we say that the subset $Q$ of $P$ has the \textsl{induced order}.
Unless stated otherwise, the term \textsl{subposet} will always mean an induced subposet.
\par
A \textsl{chain} of a poset $P$ is a subposet $C\subseteq P$ that is totally ordered.
A chain $C$ is called \textsl{maximal} 
if it is not properly contained in any larger chain of $P$. 
The \textsl{length} of a chain $C$ is defined by $\ell(C)= |C| - 1 $.
A poset $P$ is called \textsl{graded of rank $n$}
if every maximal chain of $P$ has the same length $n$.
In this case there exists a \textsl{rank function}
$\rho \colon P \rightarrow \left\{ 0,1,...,n \right\}$ 
such that $\rho(x) = 0$ for every minimal element $x \in P$,
and $\rho(y) = \rho(x) + 1$
whenever $y$ covers $x$. 
If $\rho(x) = i$, 
then we say that $x$ has rank $i$.
\par
An \textsl{order ideal} of a poset $P$ is a subset $I \subseteq P$
such that if $t \in I$ and $s \leq t$, then $s \in I$. 
The collection of all order ideals of $P$, partially ordered by inclusion, forms a poset denoted by $J(P)$.
If $P$ is finite, there is a bijection between the set of order ideals of $P$ 
and the set of antichains of $P$.
Given an antichain $A \subseteq P$,
the corresponding order ideal is 
$\left\{ x \in P \, \middle| \, x \leq a \textrm{ for some } a \in A \right\}$,
which is called the order ideal \textsl{generated} by $A$.
In particular, for $x \in P$ 
we write $\Lambda_{x} = \left\{ t \in P \, \middle| \, t \leq x \right\}$,
and call $\Lambda_{x}$ the \textsl{principal order ideal} generated by $x$.
\par
We also introduce a special type of subposet called an interval.
For elements $x, y \in P$ with $x \leq y$,
the \textsl{closed interval} 
$
[x, y]:=\left\{ t \in P \, \middle| \, x \leq t \leq y \right\}
$
is the induced subposet of $P$
 consisting of all elements lying between $x$ and $y$.
%
\begin{dfn}[\cite{ec1}]
\begin{enumerate}[leftmargin=17pt,label=(\arabic*)]
\item
For $n \in \Pos$,
where $\Pos$ denotes the set of positive integers,
let $[n]:=\{1,2,\dots,n\}$.
With the usual order, $[n]$ is a poset in which every two elements are comparable.
It is called the \textsl{chain of length $n-1$} (or the \textsl{$n$-element chain}).
\item
For $n \in \Pos$, 
let $D_n$ denote the set of all positive divisors of $n$,
partially ordered by divisibility: $i\leq j$ in $D_n$ if and only if $j$ is divisible by $i$ (denoted $i|j$).
We call $D_n$ the \textsl{poset of divisors of $n$}.
If $n=\prod_{i=1}^{r}p_{i}^{e_i}$ is the prime factorization of $n$ with $e_i>0$,
then there is an order isomorphism $D_{n}\cong [e_{1}+1]\times\cdots\times[e_{r}+1]$.
\item
Let $n\in\Non$ and let $q$ be a prime power.
Let $B_n(q)$ denote the poset of all subspaces of the $n$-dimensional vector space $\F_{q}^{n}$, ordered by inclusion. 
This poset is a lattice, called the \textsl{subspace lattice} of $\F_{q}^{n}$.
For instance Figure~\ref{fig:B32} illustrates the lattice $B_3(2)$ of the $3$-dimensional vector space over the finite field $\F_{2}$.
Here $\e_{i}$ denotes the standard basis vectors of $\F_{2}^3$,
and $\langle B\rangle$ denotes the subspace spanned by the vectors in $B$.
\end{enumerate}
\end{dfn}
\begin{figure}[htb]
\adjustbox{scale=0.6,center}{%
\begin{tikzcd}
    & & & \langle \e_1,\e_2,\e_3 \rangle & & & \\
    \langle \e_1,\e_2 \rangle \arrow{rrru} & \langle \e_1,\e_3 \rangle \arrow{rru} & \langle \e_2,\e_3 \rangle \arrow{ru} & \langle \e_1,\e_2+\e_3 \rangle \arrow{u} & \langle \e_2,\e_1+\e_3 \rangle \arrow{lu} & \langle \e_3,\e_1+\e_2 \rangle \arrow{llu} & \langle \e_1+\e_2,\e_1+\e_3 \rangle \arrow{lllu} \\
    \langle \e_1 \rangle \arrow{u}\arrow{ru}\arrow{rrru} & \langle \e_2 \rangle \arrow{lu}\arrow{ru}\arrow{rrru} & \langle \e_3 \rangle \arrow{lu}\arrow{ru}\arrow{rrru} & \langle \e_1+\e_2 \rangle \arrow{lllu}\arrow{rru}\arrow{rrru} & \langle \e_1+\e_3 \rangle \arrow{lllu}\arrow{u}\arrow{rru} & \langle \e_2+\e_3 \rangle \arrow{lllu}\arrow{llu}\arrow{ru} & \langle \e_1+\e_2+\e_3 \rangle \arrow{lllu}\arrow{llu}\arrow{lu} \\
    & & & \zero \arrow{lllu}\arrow{llu}\arrow{lu}\arrow{u}\arrow{ru}\arrow{rru}\arrow{rrru} \\
\end{tikzcd}
}
   \caption{$B_3(2)$\label{fig:B32}}
\end{figure}

\subsection{The Poset of Set Partitions}\label{ssc:SetPar}
Let $n \in \Pos$. 
We can make the set $\Pi_{n}$ of all partitions of the $n$ element set $[n]$ into a poset 
(also denoted $\Pi_{n}$) by defining $\pi\leq\sigma$ in $\Pi_{n}$ 
if every block of $\pi$ is contained in a block of $\sigma$. 
We then say that $\pi$ is a \textsl{refinement} of $\sigma$ and that $\Pi_{n}$ consists of the set partitions of $[n]$ ordered
by refinement.
We call $\pi$ a \textsl{set partition of $[n]$}
and we call $\Pi_{n}$ the \textsl{poset of set partitions of $[n]$}.
\par
More generally, if $S=\{i_{1},i_{2},\dots,i_{n}\}$ is any $n$-element set,
then we denote by $\Pi_{n}(S)$ the set of all partitions of $S$,
endowed with the same refinement order.
For example, If $S=\{1,3,4,6\}$, then $\pi=\left\{\{1\},\{4\},\{3,6\}\right\}\in\Pi_{4}(S)$ is a set partition with three blocks.
Writing $\pi_{1}=\{1\}$, $\pi_{2}=\{4\}$, and $\pi_{3}=\{3,6\}$,
we have $\pi=\{\pi_{1},\pi_{2},\pi_{3}\}$.
The partition $\pi$ is a refinement of $\sigma=\left\{\{1,4\},\{3,6\}\right\}$,
and therefore $\pi\leq\sigma$ in $\Pi_{4}(S)$.
The poset $\Pi_{n}$ is a fundamental object in algebraic combinatorics.
It is a graded lattice, where the rank of a set partition 
$\pi$ is $\rkf(\pi)=n-|\pi|$,
with $|\pi|$ denoting the number of blocks of $\pi$.
For instance, Figure~\ref{fig:Pi4} depicts $\Pi_{4}$.
\begin{figure}[hbt]
\adjustbox{scale=0.6,center}{%
\begin{tikzcd}
              &    &   & 1234\\
 1,234 \arrow{rrru} & 2,134 \arrow{rru} & 3,124 \arrow{ru} & 4,123 \arrow{u} & 12,34 \arrow{lu} & 13,24 \arrow{llu} & 14,23 \arrow{lllu} \\
 & 1,2,34  \arrow{lu}\arrow{u}\arrow{rrru}  & 1,3,24 \arrow{llu}\arrow{u}\arrow{rrru} & 1,4,23 \arrow{rrru}\arrow{u}\arrow{lllu} & 2,3,14 \arrow{lllu}\arrow{llu}\arrow{rru} & 2,4,13 \arrow{llllu}\arrow{llu}\arrow{u} & 3,4,12 \arrow{llllu}\arrow{lllu}\arrow{llu} \\
           &    &   & 1,2,3,4 \arrow{llu}\arrow{lu}\arrow{u}\arrow{ru}\arrow{rru}\arrow{rrru} \\
\end{tikzcd}
}
   \caption{$\Pi_{4}$\label{fig:Pi4}}
\end{figure}
%
\subsection{A Poset Associated with Alternating Sign Matrices}\label{ssc:posetASM}
The partially ordered set $\ASM{n}$ is introduced by J.~Striker~\cite{St}
in connection with 
the lattice structure of alternating sign matrices.
We have already introduced the poset $\ASM{n}$ in \eqref{eq:ASM-set} and \eqref{eq:ASM-order}.
This poset is graded, with rank function \cite[Proposition~3.4]{St} given by
\begin{equation}\label{eq:ASM-rank}
\rho(x,y,z)=n-2-(x+y).
\end{equation}
The cover relations in $\ASM{n}$ are described as follows.
%
\begin{prop}[\cite{St}]
In $\ASM{n}$,
an element $(x, y, z)$ covers each of the elements
\[
(x + 1, y, z),\quad (x, y + 1, z),\quad (x + 1, y, z - 1),\quad (x, y + 1, z-1),
\]
whenever these elements belong to $\ASM{n}$.
\end{prop}
\begin{figure}[htbp]
	\centering
	\begin{tikzpicture}[scale=0.8]
		\coordinate (base_x) at (225:1);
		\coordinate (base_y) at (1, 0);
		\coordinate (base_z) at (0, 1);
		\draw[->, help lines, dashed] 
			($0*(base_x) + 0*(base_y) + 0*(base_z)$)
			--
			($4*(base_x) + 0*(base_y) + 0*(base_z)$);
		\node at ($4.25*(base_x) + 0*(base_y) + 0*(base_z)$){$x$};
		\draw[->, help lines, dashed] 
			($0*(base_x) + 0*(base_y) + 0*(base_z)$)
			--
			($0*(base_x) + 4*(base_y) + 0*(base_z)$);
		\node at ($0*(base_x) + 4.25*(base_y) + 0*(base_z)$){$y$};
		\draw[->, help lines, dashed] 
			($0*(base_x) + 0*(base_y) + 0*(base_z)$)
			--
			($0*(base_x) + 0*(base_y) + 4*(base_z)$);
		\node at ($0*(base_x) + 0*(base_y) + 4.25*(base_z)$){$z$};
		\foreach \x in {0, 1, 2, 3}{
			\coordinate (v_\x_0_0) at ($\x*(base_x) + 0*(base_y) + 0*(base_z)$);
		}
		\foreach \x in {0, 1, 2}{
			\coordinate (v_\x_1_0) at ($\x*(base_x) + 1*(base_y) + 0*(base_z)$);
		}
		\foreach \x in {0, 1}{
			\coordinate (v_\x_2_0) at ($\x*(base_x) + 2*(base_y) + 0*(base_z)$);
		}
		\foreach \x in {0}{
			\coordinate (v_\x_3_0) at ($\x*(base_x) + 3*(base_y) + 0*(base_z)$);
		}
		\foreach \x in {0, 1, 2}{
			\coordinate (v_\x_0_1) at ($\x*(base_x) + 0*(base_y) + 1*(base_z)$);
		}
		\foreach \x in {0, 1}{
			\coordinate (v_\x_1_1) at ($\x*(base_x) + 1*(base_y) + 1*(base_z)$);
		}
		\foreach \x in {0}{
			\coordinate (v_\x_2_1) at ($\x*(base_x) + 2*(base_y) + 1*(base_z)$);
		}
		\foreach \x in {0, 1}{
			\coordinate (v_\x_0_2) at ($\x*(base_x) + 0*(base_y) + 2*(base_z)$);
		}
		\foreach \x in {0}{
			\coordinate (v_\x_1_2) at ($\x*(base_x) + 1*(base_y) + 2*(base_z)$);
		}
		\foreach \x in {0}{
			\coordinate (v_\x_0_3) at ($\x*(base_x) + 0*(base_y) + 3*(base_z)$);
		}
		\foreach \x/\xx in {3/2, 2/1, 1/0}{
			\draw[directed] (v_\x_0_0) -- (v_\xx_0_0);
		}
		\foreach \x/\xx in {2/1, 1/0}{
			\draw[directed] (v_\x_1_0) -- (v_\xx_1_0);
		}
		\foreach \x/\xx in {1/0}{
			\draw[directed] (v_\x_2_0) -- (v_\xx_2_0);
		}
		\foreach \y/\yy in {3/2, 2/1, 1/0}{
			\draw[directed] (v_0_\y_0) -- (v_0_\yy_0);
		}
		\foreach \y/\yy in {2/1, 1/0}{
			\draw[directed] (v_1_\y_0) -- (v_1_\yy_0);
		}
		\foreach \y/\yy in {1/0}{
			\draw[directed] (v_2_\y_0) -- (v_2_\yy_0);
		}
		\foreach \x/\xx in {2/1, 1/0}{
			\draw[directed] (v_\x_0_1) -- (v_\xx_0_1);
		}
		\foreach \x/\xx in {1/0}{
			\draw[directed] (v_\x_1_1) -- (v_\xx_1_1);
		}
		\foreach \y/\yy in {2/1, 1/0}{
			\draw[directed] (v_0_\y_1) -- (v_0_\yy_1);
		}
		\foreach \y/\yy in {1/0}{
			\draw[directed] (v_1_\y_1) -- (v_1_\yy_1);
		}
		\foreach \x/\xx in {1/0}{
			\draw[directed] (v_\x_0_2) -- (v_\xx_0_2);
		}
		\foreach \y/\yy in {1/0}{
			\draw[directed] (v_0_\y_2) -- (v_0_\yy_2);
		}
		\foreach \x/\z/\xx/\zz in {3/0/2/1, 2/1/1/2, 1/2/0/3}{
			\draw[directed] (v_\x_0_\z) --(v_\xx_0_\zz);
		}
		\foreach \x/\z/\xx/\zz in {2/0/1/1, 1/1/0/2}{
			\draw[directed] (v_\x_1_\z) --(v_\xx_1_\zz);
		}
		\foreach \x/\z/\xx/\zz in {1/0/0/1}{
			\draw[directed] (v_\x_2_\z) --(v_\xx_2_\zz);
		}
		\foreach \y/\z/\yy/\zz in {3/0/2/1, 2/1/1/2, 1/2/0/3}{
			\draw[directed] (v_0_\y_\z) --(v_0_\yy_\zz);
		}
		\foreach \y/\z/\yy/\zz in {2/0/1/1, 1/1/0/2}{
			\draw[directed] (v_1_\y_\z) --(v_1_\yy_\zz);
		}
		\foreach \y/\z/\yy/\zz in {1/0/0/1}{
			\draw[directed] (v_2_\y_\z) --(v_2_\yy_\zz);
		}
		\foreach \x/\z/\xx/\zz in {2/0/1/1, 1/1/0/2}{
			\draw[directed] (v_\x_0_\z) --(v_\xx_0_\zz);
		}
		\foreach \x/\z/\xx/\zz in {1/0/0/1}{
			\draw[directed] (v_\x_1_\z) --(v_\xx_1_\zz);
		}
		\foreach \y/\z/\yy/\zz in {2/0/1/1, 1/1/0/2}{
			\draw[directed] (v_0_\y_\z) --(v_0_\yy_\zz);
		}
		\foreach \y/\z/\yy/\zz in {1/0/0/1}{
			\draw[directed] (v_1_\y_\z) --(v_1_\yy_\zz);
		}
		\foreach \x/\z/\xx/\zz in {1/0/0/1}{
			\draw[directed] (v_\x_0_\z) --(v_\xx_0_\zz);
		}
		\foreach \y/\z/\yy/\zz in {1/0/0/1}{
			\draw[directed] (v_0_\y_\z) --(v_0_\yy_\zz);
		}
		\foreach \x in {0, 1, 2, 3}{
			\fill[draw=black, fill=white] (v_\x_0_0) circle (2pt);
		}
		\foreach \x in {0, 1, 2}{
			\fill[draw=black, fill=white] (v_\x_1_0) circle (2pt);
		}
		\foreach \x in {0, 1}{
			\fill[draw=black, fill=white] (v_\x_2_0) circle (2pt);
		}
		\foreach \x in {0}{
			\fill[draw=black, fill=white] (v_\x_3_0) circle (2pt);
		}
		\foreach \x in {0, 1, 2}{
			\fill[draw=black, fill=white] (v_\x_0_1) circle (2pt);
		}
		\foreach \x in {0, 1}{
			\fill[draw=black, fill=white] (v_\x_1_1) circle (2pt);
		}
		\foreach \x in {0}{
			\fill[draw=black, fill=white] (v_\x_2_1) circle (2pt);
		}
		\foreach \x in {0, 1}{
			\fill[draw=black, fill=white] (v_\x_0_2) circle (2pt);
		}
		\foreach \x in {0}{
			\fill[draw=black, fill=white] (v_\x_1_2) circle (2pt);
		}
		\foreach \x in {0}{
			\fill[draw=black, fill=white] (v_\x_0_3) circle (2pt);
		}
		\node at ($3*(base_x) - 0.5*(base_y) + 0*(base_z)$){$3$};
		\node at ($0.5*(base_x) + 3.5*(base_y) + 0*(base_z)$){$3$};
		\node at ($0*(base_x) + 0.5*(base_y) + 3*(base_z)$){$3$};
		\node at ($0*(base_x) - 1*(base_y) + 3*(base_z)$){$(0, 0, 3)$};
		\node at ($1*(base_x) - 1*(base_y) + 2*(base_z)$){$(1, 0, 2)$};
	\end{tikzpicture}
	\caption{$\ASM{5}$}
	\label{fig:AA_5}
\end{figure}
We illustrate $\ASM{5}$ in Figure~\ref{fig:AA_5}.
In this figure, an open circle represents an element of $\ASM{5}$,
and the arrows represent the cover relations.
For example,
$(0, 0, 3)$ covers $(1, 0, 2)$ in $\ASM{5}$,
that is, there is an arrow from $(1, 0, 2)$ to $(0, 0, 3)$.
\par
The poset $\ASM{n}$
can equivalently be described as a layering of successively smaller type $A$ positive root posets; see \cite{SW}.
%
%

%
%
%
%
%
\section{Sprague-Grundy Theory of Coin Turning Games}\label{sc:PosetGame}
In this section, we describe a finite impartial game induced from the poset structure of a given poset; see \cite{Lenstra,Sato}.
The key theorem in this section is Theorem~\ref{thm:FT_of_T_game},
which provides a method to compute the Grundy functions for games of this type,
and frequently used in the rest of the paper.
\par
We also introduce the symmetric difference of two sets.
Let $X$ be a set,
and let $A,B\subseteq X$.
We define the set
$A \ominus B:=\left( A \cup B \right) \setminus \left( A \cap B \right)$,
and call it the \textsl{symmetric difference} of $A$ and $B$.
The symmetric difference is a commutative operation, and it is also associative.
\subsection{Coin Turning Games}\label{ssc:XTgame}
Throughout this section, we fix a finite poset $X$ 
and a family $\T$ of subsets of $X$
which satisfies 
\begin{quote}
		($\sharp$)\qquad each $T \in \T$ has a (unique) maximum element.
\end{quote}
For each $T\in\T$,
we denote by $m_{T}$ the unique maximal element of $T$.
For $x\in X$,
we define $\T_{x}=\{T\in\T\mid m_{T}=x\}$,
and for a subset $P\subseteq X$,
we define $\T_{P}=\{T\in\T\mid m_{T}\in P\}$.
%
\begin{dfn}[Coin Turning Game \cite{Lenstra,Sato}]\label{def:PosetGame}
	Let $X$ be a finite poset,
	and $\T$ a family of subsets of $X$
	which satisfies the condition~($\sharp$).
	We define a game 
	$\Game$ 
	by the following rule:
\begin{enumerate}[label=(\alph*)] 
\item\label{it:PG01}
	The set of positions is $\P = {2}^{X}$,
	the collection of all subsets of $X$.
\item\label{it:PG02}
	Given a position $P\in\P$, 
	the player in turn can choose a subset $T \in \T_{P}$ 
	and make the move $P\to P \ominus T$,
\item\label{it:PG03}
	A position $P$ is an ending position if $\T_{P}=\emptyset$.
	Equivalently, the set of ending positions is
			$
				\epsilon
				=
				\left\{
					P \in \P
				\, \middle| \,
					{m}_{T} \notin P
					\textrm{ for any }
					T \in \T
				\right\}
			$.
\end{enumerate} 
	We call this game the \textsl{$(X,\T)$-game} (or simply \textsl{$\T$-game}),
	and we denote it by $\Game=\Pgame(X, \T)$.
\end{dfn}
We refer to games of this type as \textsl{(finite) coin turning game}.
Let $D:=\left\{m_{T}\,\middle|\, T\in \T\right\}$ and $E:=X\setminus D$.
Then a position $P\in \P$ is an ending positions if and only if there is no element $P \cap D=\emptyset$,
that is, if and only if $P\subseteq E$.
Consequently, $\epsilon=2^E$.
\par\smallskip
We will see that no position can occur more than once during a play of the game.
In particular, from any given position only finitely many positions are reachable.
To establish this fact, we prove the following proposition.
%
\begin{prop}
Let $X$ be a finite poset and let $\T$ be a family of subsets of $X$ satisfying condition $(\sharp)$.
Fix a linear extension $\tau: X \to C$, where $C=[n]$ is the chain with $n=|X|$.
Let $\rho:C\to \Non$ be the rank function of $C$,
defined by $\rho(i)=i-1$  for $i\in[n]$.
Define a map $F:\P=2^X\to \Non$ by
\[
F(P)=\sum_{x\in P}2^{\rho(\tau(x))}.
\]
Thus $F$ assigns to each position $P\in\P$ a non-negative integer.
Then,
for any position $P$ and any option $P'\in N(P)$, we have $F(P')<F(P)$.
\end{prop}
\begin{demo}{Proof}
If $P'\in N(P)$, then there exists $T\in\T$ such that its unique maximal element $x=m_{T}\in P$ satisfies $P'=P\ominus T$.
In particular, we have $x\in P$, but $x\not\in P'$.
Moreover, for any element $y\in P'\setminus P$,
condition ($\sharp$) implies that $y<x$.
Since $\tau$ is a linear extension, this yields $\rho(\tau(y))<\rho(\tau(x))$.
Consequently, in passing from $P$ to $P'$,
the term $2^{\rho(\tau(x))}$ is removed from the sum defining $F(P)$,
while any newly added terms correspond to strictly smaller powers of $2$.
Hence we conclude that $F(P')<F(P)$.
\end{demo}
This proposition means that this function $F(P)$ strictly decreases at each move and therefore never remains constant or increases.
Consequently, no position can appear more than once during a play, and every play terminates after finitely many moves.
\par\smallskip
The reason we call the $(X,\T)$-game a coin-turning game is as follows.
Imagine a coin is placed on each element of $X$, and each coin shows either heads or tails.
For a position $P\in\P$, the coin at a vertex $x\in X$ shows heads if $x\in P$, and shows tails otherwise.
On a player's turn, the player chooses a set $T \in \T_{P}$ 
and flips all coins on the vertices in $T$.
If a player cannot find any $T \in \T_{P}$, then that player loses.
From this perspective, we refer to the poset $X$ as the \textsl{board}, and to each subset 
$T \in \T$ as a \textsl{turning set}.
\subsection{The Fundamental Theorem of Coin Turning Games}
We call the following as the fundamental theorem of coin turning game.
It plays a key role in our paper.
%
\begin{thm}[Fundamental Theorem of Coin Turning Games \cite{Lenstra}]\label{thm:FT_of_T_game}
	\begin{subequations}
	Let $X$ be a finite poset,
	and let $\T$ be a family of subsets of $X$ satisfying condition ($\sharp$).
	Define a map 
	$g_{X,\T}:X \rightarrow \Non$
	inductively by
	\begin{equation}\label{eq:FT_of_T_game_1}
		g_{X,\T}(x):=
		\mex
		\Bigg\{
			\underset{t \in T\setminus\{x\}}{\,\nimsum}g_{X,\T}(t)
		\, \Bigg| \,	
			T \in \T_{x}
		\Bigg\}.
	\end{equation}
	Recall that ${\nimsum}_{t\in\emptyset}t=0$.
	For any subset $P \subseteq X$, we set
	\begin{equation}\label{eq:FT_of_T_game_2}
		g_{X,\T} \left( P \right)
		=
		\underset{x \in P}{\,\,\nimsum} 
			\, 
			g_{X,\T}(x)
			.
	\end{equation}
	Then the map $g_{X,\T}:2^{X}\to \Non$ agrees with the Grundy function $g_{\Game}$ of the game $\Game=\Pgame(X,\T)$.
	When no confusion can arise, we write $g_{\T}$ or simply $g$ in place of $g_{X,\T}$.
	\end{subequations}
\end{thm}
\begin{demo}{Proof}
We proceed by induction on the length $\ell(P)$.
\par\smallskip
If $\ell(P)=0$ then $P\subseteq E$, so $\T_{P}=\emptyset$.
Hence $g_{X,\T}(x)=\mex(\emptyset)=0$ for all $x\in P$,
and therefore $g_{X,\T}(x)=\underset{x \in P}{\nimsum}g_{X,\T}(x)=0$, 
which agrees with the Grundy value of an ending position.
\par\smallskip
Now assume $\ell(P)>0$.
By definition,
\[
g_{X,\T}(x)=\mex\biggl(\biggl\{\underset{t\in T\setminus\{x\}}{\nimsum}g_{X,\T}(t) \,\bigg|\, T\in \T_{x}\biggr\}\biggr).
\]
Using the standard property of the $\mex$ operator with Nim addition
(cf. \eqref{eq:NIM_add_and_mex_multi}),
we see that the right-hand side of \eqref{eq:FT_of_T_game_1} equals
\begin{align*}
\underset{x \in P}{\nimsum} g_{X,\T}(x)
&=\mex\biggl(
\bigcup_{x \in P}
\bigcup_{T\in\T_{x}}
\biggl\{
\underset{y\in P\setminus \{x\}}{\nimsum} g_{X,\T}(y)\nimadd \underset{t\in T\setminus\{x\}}{\nimsum}g_{X,\T}(t)
\biggr\}
\biggr)
\\
&=\mex\biggl(
\bigcup_{T\in\T_{P}}
\biggl\{
\underset{y\in P\setminus \{m_{T}\}}{\nimsum} g_{X,\T}(y)\nimadd \underset{t\in T\setminus\{m_{T}\}}{\nimsum}g_{X,\T}(t)
\biggr\}
\biggr).
\end{align*}
If $y \in P\cap T$ then $g_{X,\T}(y)$ is cancelled by the same term in the second sum.
Hence we obtain
\[
\underset{x \in P}{\nimsum} g_{X,\T}(x)
=\mex\biggl\{\underset{y\in P\ominus T}{\nimsum} g_{X,\T}(y)
\mid T\in \T_{P}
\biggr\}.
\]
Since $\ell(P\ominus T)<\ell(P)$,
we have 
$\underset{y\in P\ominus T}{\nimsum} g_{X,\T}(y)=g_{X,\T}(P\ominus T)$
by induction hypothesis.
This shows that $\underset{x \in P}{\nimsum} g_{X,\T}(x)$ is the Grundy value of $P$
from Definition~\ref{def:Grundy}~\ref{it:Grundy2} and Definition~\ref{def:PosetGame}~\ref{it:PG02}.
This completes the proof.
\end{demo}
%
\subsection{The Product Theorem for Coin Turning Games}
Let $\Game_{1}=\Pgame(X_{1}, \T_{1})$ and $\Game_{2}=\Pgame(X_{2}, \T_{2})$ be coin turning Games, where
 $X_{1}$ and $X_{2}$ are finite posets and $\T_{1}$ and $\T_{2}$ are their families of turning sets, respectively.
Let $X=X_{1}\times X_{2}$ be the direct product poset.
Define
\[
\T=\{T_{1}\times T_{2} \,|\, T_{1}\in\T_{1},\,  T_{2}\in\T_{2}\}.
\]
Then, for each $T_{1}\times T_{2}\in \T$, the element $(m_{T_{1}},m_{T_{2}})$ is the unique maximum element of $T_{1}\times T_{2}$.
Hence the family $\T$ satisfies condition~($\sharp$).
%
\begin{thm}[Product Theorem for Coin Turning Games \cite{Lenstra}]\label{th:nim-prod-app}
\begin{subequations}
Then the Grundy function $g_{X,\T}$ is given by
\begin{equation}
g_{X,\T}(x_{1},x_{2})=g_{X_{1},\T_{1}}(x_{1}) \nimmul g_{X_{2},\T_{2}}(x_{2})\label{eq:nim-prod-app1} 
\end{equation}
for all $(x_{1},x_{2})\in X$.
Consequently, for arbitrary subsets $P_{1}\subseteq X_{1}$ and $P_{2}\subseteq X_{2}$,
we have
\begin{equation}\label{eq:nim-prod-app2}
g_{X,\T}(P_{1}\times P_{2})=\underset{(x_{1},x_{2})\in P_{1}\times P_{2}}{\nimsum}
g_{X,\T}(x_{1},x_{2}).
\end{equation}
\end{subequations}
\end{thm}
This follows from the Sprague-Grundy theorem for sums of impartial games, 
together with the definition of Nim multiplication as the Grundy value of the Cartesian product of two positions.
This product theorem will be used for divisor posets (see Theorem~\ref{th:Grundy-Dn-ruler}) 
and for the poset of set partitions (see Theorem~\ref{th:SetParGrundy}, \ref{it:Pi-2}).
%
%
%
\subsection{Two Typical Classes of Coin Turning Games}
As described above,
a coin turning game $\Game=\Pgame(X,\T)$ depends not only on the underlying poset $X$
but also on the choice of the turning set $\T$.
We first present three typical choices for the turning set.
\begin{dfn}
Let $X$ be a poset.
\begin{enumerate}[leftmargin=20pt,label=(\roman*)] 
	\item\label{item:TT}
	If we take
	$\TT=\left\{\{ x, y \}\, \middle| \,x, y \in X,\,x \leq y \right\}$ 
	as the turning set,
	then the game $\Pgame\left( X, \TT\right)$ is called the \textsl{turning turtles} on $X$.
	\item\label{item:OIG}
	If we take the turning set to be $\OI=\left\{\Lambda_{x} \, \middle| \, x \in X\right\}$,
	the family of all principal order ideals of $X$,
	of all principal order ideals of $X$  as the turning set,
	then the game $\Pgame(X,\OI)$ is called the \textsl{order ideal game} on $X$.
	\item\label{item:Ruler}
	If we take the turning set to be 
	$\I=\left\{[x,y]\, \middle| \, x,y\in X,\, x\leq y\right\}$ 
	the family of all intervals of $X$, then the game $\Pgame(X,\I)$ is called the \textsl{ruler} on $X$.
\end{enumerate} 
\end{dfn}
%
%
In what follows, cases~\ref{item:OIG} and~\ref{item:Ruler} constitute the primary objects of study, 
since the problem is usually straightforward in case~\ref{item:TT}.
Accordingly, the main objective of this paper is 
to compute the Sprague-Grundy functions of the coin turning games $\Pgame(X,\T)$ 
in cases~\ref{item:OIG} and~\ref{item:Ruler}.
Among these, the case~\ref{item:Ruler} is substantially more challenging.
\par\smallskip
%
%
The following proposition is elementary and an immediate consequence of the definitions together with Theorem~\ref{thm:FT_of_T_game},
but extremely useful for simplifying the computation of Grundy values by exploiting symmetries of posets.
It will be used to prove Lemma~\ref{lem:g_value_isomorphic} and Lemma~\ref{lem:isomorphism_eta} 
for the ASM posets.
%
\begin{prop}\label{prop:g_value_order_isomorphism_general}
	Let $X_{1}$ and $X_{2}$ be finite posets, and let $\T_{1}$ and $\T_{2}$ be families of turning sets on $X_{1}$ and $X_{2}$, respectively.
	Suppose that $f:X_{1}\to X_{2}$ is an order-preserving bijection such that
\begin{equation}\label{eq:iso-T}
\T_{2}=\{f\left(T\right)\mid T\in \T_{1}\}.
\end{equation}
Then the Grundy functions satisfy
	$
		g_{\T_{1}}\left( x \right) 
		= 
		g_{\T_{2}}\left( f(x) \right)
	$
	for all $x\in X_{1}$.
	Hence $g_{\T_{1}}\left( P \right)=g_{\T_{2}}\left( f(P) \right)$ for any subset $P\subseteq X_{1}$.
\end{prop}
\subsection{Order Ideal Games}
\par\smallskip
In the order ideal game,
we obtain the following lemma as a direct application of Theorem~\ref{thm:FT_of_T_game}.
Although this lemma is simple, it plays a crucial role in \S~\ref{sbsec:ASM-ideal} 
in the proof of Theorem~\ref{th:ASM-ideal}
concerning order ideal games on the ASM posets.
%
\begin{lem}\label{lem:OrderI01}
	Let $X$ be a finite poset.
	In the order ideal game on $X$,
	the Grundy function
	$g_{\OI}:X \to \Non$
	is determined by
	\begin{equation}\label{eq:OIgame01}
		g_{\OI}\left( x \right)
		=
		\begin{cases}
			0	
			&	
			\textrm{if } 
			\left|
			\left\{ 
				t \in {\Lambda}_{x} \setminus \{ x \}
			\, \middle| \,
				g_{\OI}\left( t \right)
				= 
				1
			\right\}\right|
			\text{ is odd}
			, \\
			1	
			&
			\textrm{if } 
			\left|
			\left\{ 
				t \in {\Lambda}_{x} \setminus \{ x \}
			\, \middle| \,
				g_{\OI}\left( t \right)
				= 
				1
			\right\}\right|
			\text{ is even.}
		\end{cases}
	\end{equation}
	Thus $g_{\OI}\left( x \right)$ takes value only $0$ or $1$
	for each $x \in X$.
	Consequently, for any position $P \subset X$,
	the Grundy value 
	$
		g_{\OI}\left( P \right)
	$ 
	is also restricted to $0$ or $1$.
\end{lem}
\begin{demo}{Proof}
For each position of the form $P=\{ x \}$,
there is exactly one available move, namely $\{x\}\to{\Lambda}_{x} \setminus \{ x \}$.
Therefore, \eqref{eq:FT_of_T_game_1} of 
Theorem~\ref{thm:FT_of_T_game} immediately yields the formula \eqref{eq:OIgame01}.
The second claim follows directly from this equation.
\end{demo}
%
\begin{cor}\label{cor:OrderI}
Let $X$ be a graded poset with a unique minimum element $\hat0$,
and let $\rkf:X\to\Non$ denote its rank function.
Then the Grundy function of the order ideal game $\Pgame\left( X, \OI \right)$ is given by
\begin{equation}
g_{\OI}(x)=\begin{cases}
1,&\text{ if $x=\hat0$,}\\
0,&\text{ otherwise,}
\end{cases}
\end{equation}
for all $x\in X$.
\end{cor}
\begin{demo}{Proof}
We proceed by induction on the rank $\rkf(x)$.
If $\rkf(x)=0$, that is, $x=\hat0$, then $\Lambda_{x}\setminus\{x\}$ is empty, and hence
$g(x)=\mex\{ 0\}=1$.
\par
Suppose now that $\rkf(x)>0$. 
By the induction hypothesis, 
$\hat0$ is the only element in $\Lambda_{x}\setminus\{x\}$ whose Grundy value equals $1$.
Therefore, equation~\eqref{eq:OIgame01} in Lemma~\ref{lem:OrderI01} yields the desired conclusion.
This completes the proof.
\end{demo}
Hence, if the poset $X$ is graded and has $\hat0$, 
then the Grundy function has a particularly simple form, as described above.
The ASM poset $\ASM{n}$ is graded but has multiple minimal elements,
and therefore determining the Grundy function for the order ideal game on $\ASM{n}$
becomes a very interesting and nontrivial problem.
%

%
%
%
%
%
\section{Sprague-Grundy Analysis of Coin Turning Games}\label{sc:SG-Analysis}
The purpose of this section is to determine the Grundy functions of several coin turning games.
\par
Let $a, b\in\Non$.
The operation $a\oplus b$ is performed by writing \(a\) and \(b\) in binary and then adding the digits without carrying; e.g., 
$5\oplus9=(101)_{2}\oplus (1001)_{2}=(1100)_{2}=12$.
It is well-known that $(\Non,\oplus)$ forms an abelian group of exponent $2$.
In Theorem~\ref{thm:nimaddd=oplus}, we prove that the NIM-sum \(a\nimadd b\) defined in Definition~\ref{def:Nimadd} coincides with \(a\oplus b\).
\begin{dfn}\label{def:binary}
Let $i$ be a nonnegative integer.
For a positive integer $x$,
we denote by $\digit{i}{x}$ the $i$th digit in the binary expansion of $x$.
In other words we have $x=\sum_{i\geq0}\digit{i}{x}2^{i}$ where $\digit{i}{x}=0\text{ or }1$.
We further use the notation
$\placemax{x}=\max\{i\in\Non \,|\, \digit{i}{x}\neq 0\}$ and 
$\placemin{x}=\min\{i\in\Non \,|\, \digit{i}{x}\neq 0\}$.
For a positive integer \(x\), let $\fnh{x}=2^{\placemin{x}}$.
The sequence $\placemin{x}$ for $x\in\Pos$ is well known as the \textsl{ruler sequence}.
By a slight abuse of language, we also refer to the sequence $\fnh{x}$ as the ruler sequence.
\end{dfn}
For example,
if $x=26=(11010)_{2}$ then we write $\digit{4}{x}=1$, $\digit{0}{x}=0$, $\placemax{x}=4$,
 $\placemin{x}=1$ and hence $\fnh{x}=2$.
The first few values of $\placemin{x}$ and $\fnh{x}$ are listed in Table~\ref{tbl:phi}.
\begin{table}[h]
 \begin{center}
	\begin{tabular}{|l||c|c|c|c|c|c|c|c|c|c|c|c|c|c|c|c|} \hline
$x$            &   1 &  2 &  3 &  4 &  5 &  6 &  7 &  8 &  9 & 10 & 11 & 12 & 13 & 14 & 15 \\\hline
$\placemin{x}$ &   0 &  1 &  0 &  2 &  0 &  1 &  0 &  3 &  0 &  1 &  0 &  2 &  0 &  1 &  0 \\\hline
$\fnh{x}$      &   1 &  2 &  1 &  4 &  1 &  2 &  1 &  8 &  1 &  2 &  1 &  4 &  1 &  2 &  1 \\\hline
	\end{tabular}
   \caption{$\fnh{x}$\label{tbl:phi}}
 \end{center}
\end{table}
%
%
\subsection{Rulers on Finite Chains}
In this subsection we begin with the simplest case,
namely, $X=C=[n]$, the $n$-element chain.
The Grundy function for the order ideal game on \(C\) is given by Corollary~\ref{cor:OrderI}
since $C$ is graded and has a minimum element $\hat0$.
Hence, in this subsection we describe the Grundy function of the ruler on \(C\). 
Although this case is well known,
we include a proof for the sake of completeness.
The key ingredient is the following lemma,
which characterizes the ruler sequence via the minimum excluded value.
%
%
\begin{lem}\label{lem:chain-ruler-recurrence}
Let $f:\Pos\to\Non$ which satisfies $f(1)=1$, and
\begin{equation}\label{eq:chain-ruler-recurrence}
f(n)=\mex\Bigl\{\overset{n-1}{\underset{k=m}{\,\,\,\nimsum}}f(k) \,|\, m=1,2,\dots,n\Bigr\}
\end{equation}
for $n>1$.
Then we have $f(n)=\fnh{n}$.
\end{lem}
We give a proof of this lemma in Appendix~\ref{sc:RC}.
This lemma will also be used in \S~\ref{sc:Divisor} and \S~\ref{sc:Subspcace}.
Hence we are in the position to state the main theorem of this section.
%
\begin{thm}[\cite{Sato}]\label{th:SGf-Chain-ruler}
Let $n$ be a positive integer,
and let $C=[n]$ be the $n$-element chain with the rank function $\rkf(x)=x-1$.
Then the Grundy function of the ruler on $C$ is given by
\begin{equation}
g(x)=\fnh{x}=\fnh{\rho(x)+1}
\end{equation}
for all $x\in [n]$.
\end{thm}
%
%
\begin{demo}{Proof}
The proof of this theorem reduces to Theorem~\ref{thm:FT_of_T_game} together with the above lemma.
Let $g$ denote the Grundy function of the ruler $\Pgame(C,\I)$.
\par
If $x=1\in C=[n]$ is the minimum element of $C$,
there is no element less than $1$.
Since empty NIM-sum is $0$,
we obtain $g(1)=\mex\{0\}=1$ by \eqref{eq:FT_of_T_game_1}.
\par
Now let $x\in C$ with $x>1$. 
Then $\T_{x}$ consists of the intervals of the form $[m,x]$ for $m=1,2,\dots,x$.
By \eqref{eq:FT_of_T_game_1},
this immediately implies
\begin{equation}\label{eq:rec-chain}
g(x)=\mex\Bigl\{\overset{x-1}{\underset{k=m}{\,\,\,\nimsum}}g(k) \,|\, m=1,2,\dots,x\Bigr\}.
\end{equation}
Hence, $g(x)$ satisfies the recurrence relation of Lemma~\ref{lem:chain-ruler-recurrence}.
Accordingly, we conclude that $g(n)=\fnh{n}$ for all $n\in\Pos$.
%
%
\end{demo}
%
%
%
\subsection{Rulers on Divisor Posets}\label{sc:Divisor}
Let $r$ be a nonnegative integer,
and let $n=\prod_{i=1}^{r}p_{i}^{e_{i}}$ be a positive integer, where each $p_{i}$ is a prime number and $e_{i}\in\Pos$ for all $i$.
The lattice $D_n$ has the minimul element $\hat0=1$ and the maximal element $\hat1=n$.
Hence, as before,
the Grundy functions of the order ideal game $\Pgame\left(D_{n},\OI\right)$ on $D_{n}$ 
are given by Corollay~\ref{cor:OrderI}.
Furthermore, this poset has the structure $D_n\simeq [e_{1}+1]\times\cdots\times[e_{r}+1]$,
that is, it can be written as the direct product of $r$ chains.
Hence we can apply Theorem~\ref{th:nim-prod-app} to obtain the following result:
%
\begin{thm}\label{th:Grundy-Dn-ruler}
Let $e$, $n$, $p_i$ and $e_i$ ($i=1,\dots,r$) be as above.
The Grundy values of the ruler on the lattice $D_n$ is given by
\begin{equation}
g(y)=g(p_{1}^{x_{1}}\cdots p_{r}^{x_{r}})
=\fnh{x_{1}+1}\nimmul\cdots\nimmul \fnh{x_{r}+1}
\end{equation}
for $y=p_{1}^{x_{1}}\cdots p_{r}^{x_{r}}\in D_{n}$.
\end{thm}
When $n=12=2^2\cdot3$, the Grundy values of 
the order ideal game and the ruler on $D_{12}$ are shown in
 Figure~\ref{fig:D12-OI} and Figure~\ref{fig:D12-ru}, respectively.

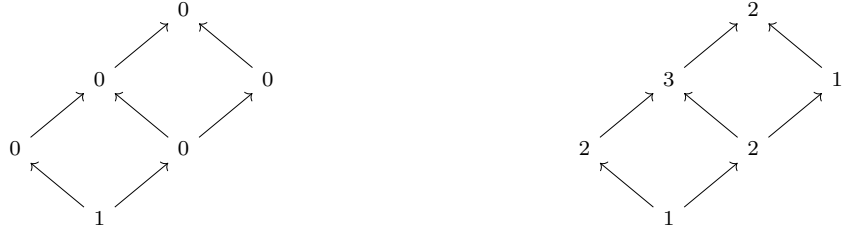
\begin{figure}[htbp]
{\scriptsize
\begin{subfigure}[b]{0.4\textwidth}
    \centering
\begin{tikzcd}
              &                       & 0 \\
              & 0 \arrow{ru}          &                       & 0 \arrow{lu}  \\
 0 \arrow{ru} &                       & 0  \arrow{lu}\arrow{ru}   \\
              & 1 \arrow{lu}\arrow{ru}\\
\end{tikzcd}
    \caption{Order ideal game on $D_{12}$}
    \label{fig:D12-OI}
\end{subfigure}
}
\hfill
{\scriptsize
\begin{subfigure}[b]{0.4\textwidth}
    \centering
\begin{tikzcd}
              &                       & 2 \\
              & 3 \arrow{ru}          &                       & 1 \arrow{lu}  \\
 2 \arrow{ru} &                       & 2  \arrow{lu}\arrow{ru}   \\
              & 1 \arrow{lu}\arrow{ru}\\
\end{tikzcd}
    \caption{Ruler on $D_{12}$}
    \label{fig:D12-ru}
\end{subfigure}
}
    \caption{Poset $D_{12}$ of divisors of $12$}
\end{figure}
%
%
%
\subsection{Ruler on Subspace Lattice over a Finite Field}\label{sc:Subspcace}
Let $n$ be a nonnegative integer and let $q$ be a prime power.
In this subsection, we consider the case $X=B_n(q)$ and the order is defined by inclusion.
Then $X$ is a finite graded poset with rank function $\rho(W)=\dim W$ for $W\in X$,
and with unique minimum element $\hat0=\zero=\{0\}$.
Consequently, Corollary~\ref{cor:OrderI}
apply, yielding the Grundy functions for the order ideal game on $X$.
\par
The aim of this subsection is to establish the following theorem describing 
the Grundy function of the ruler on $X$.
%
\par
Let $\qbinom{n}{r}{q}:=\frac{(q;q)_{n}}{(q;q)_{r}(q;q)_{n-r}}$ denote the $q$-binomial coefficient,
where $(a;q)_{k}=\begin{cases}\prod_{i=1}^{k}(1-aq^{i-1}),&\text{ if $k\geq 0$,}\\ 1/(a;q)_{-k},&\text{ if $k<0$,}\end{cases}$ is the $q$-Pochhammer symbol.
It is also well konwn that the $q$-binomial coefficients satisfy the recurrence relations
\begin{equation}\label{eq:q-binom-rec}
\qbinom{n}{r}{q}=\qbinom{n-1}{r-1}{q}+q^{r}\qbinom{n-1}{r}{q}=q^{n-r}\qbinom{n-1}{r-1}{q}+\qbinom{n-1}{r}{q}
\end{equation}
for $n\geq 1$, together with the initial condition $\qbinom{0}{r}{q}=\delta_{0,r}$,
where $\delta_{i,j}$ denotes the Kronecker delta.
%
\begin{thm}\label{th:Grundy-subsFF-ruler}
\begin{subequations}
Let \(n\) be a nonnegative integer, and let \(q\) be a prime power.
Let $W\in B_n(q)$ be a subspace of $\F_{q}^{n}$.
The Grundy value of $W$ in $B_n(q)$ depends only on the dimension $\dim W$;
hence it is equal to the Grundy value of $\F_{q}^d$ in $B_d(q)$,
which is denoted by $g_{q}(d)$.
\begin{enumerate}[label=(\roman*)] 
	\item\label{item:q-even}
	If $q$ is even, then
	\begin{equation}
	g_{q}(d)=\fnh{d+1}.
	\end{equation}
	\item\label{item:q-odd}
	If $q$ is odd, then 
	\begin{equation}
	g_{q}(d)=\MOD(d,3)+1.
	\end{equation}
\end{enumerate} 
Here, $\MOD(x,m)$ denotes the remainder when \(x\) is divided by \(m\).
\end{subequations}
\end{thm}
The first few values of $g_{q}(d)$ are listed in Table~\ref{tb:gq(d)}.
The Grundy values corresponding to each element of $B_3(2)$ in Figure~\ref{fig:B32}
are listed in Figure~\ref{fig:B32g-ruler}.
Hereafter we use the notation $G_{q}(n,r):=\MOD\Bigl(\qbinom{n}{r}{q},2\Bigr)$.
For nonnegative integers $m$ and $a$,
the product $ma$ denotes $\underbrace{a\nimadd\cdots\nimadd a}_{m\text{ times}}$.
The following lemma gives recurrence relations for $g_{q}(d)$.
\begin{table}[htb]
 \begin{center}
   \caption{The Grundy values $g_{q}(d)$}\label{tb:gq(d)}
	\begin{tabular}{|l||c|c|c|c|c|c|c|c|c|c|c|c|c|c|c|c|} \hline
 $d=\dim W$ &   0 &  1 &  2 &  3 &  4 &  5 &  6 &  7 &  8 &  9 & 10 & 11 & 12 & 13 & 14 \\\hline\hline
 $g_{q}(d)$ ($q$ even) &   1 &  2 &  1 &  4 &  1 &  2 &  1 &  8 &  1 &  2 &  1 &  4 &  1 &  2 &  1 \\\hline
 $g_{q}(d)$ ($q$ odd)  &   1 &  2 &  3 &  1 &  2 &  3 &  1 &  2 &  3 &  1 &  2 &  3 &  1 &  2 &  3 \\\hline
	\end{tabular}
 \end{center}
\end{table}
%
%
\begin{lem}\label{lem:Bn(q)-recurrence}
\begin{subequations}
Let \(d\) be a nonnegative integer, and let \(q\) be a prime power.
\begin{enumerate}[label=(\arabic*),leftmargin=2em]
\item\label{it:Bn(q)lem1}
Let $W\in B_n(q)$.
The Grundy value of $W$ in the ruler on $B_n(q)$ depends only on its dimension $d=\dim W$.
Accordingly, we denote the Grundy value of $W$ by $g_{q}(d)$.
\item\label{it:Bn(q)lem2}
For $m=0,1,\dots,d$, define
\begin{equation}\label{eq:rec:sq(d,m)-general}
s_{q}(d,m):=\overset{d-1}{\underset{k=m}{\,\,\,\nimsum}}G_{q}(d-m,k-m)\,g_{q}(k).
\end{equation}
Then
\begin{equation}\label{eq:rec-mex}
g_{q}(d)=\mex\left\{\,s_{q}(d,m) \,\mid\, m=0,1,\dots,d\,\right\}.
\end{equation}
\end{enumerate}
\end{subequations}
\end{lem}
\begin{demo}{Proof}
Let $X=B_n(q)$ and let $\I$ denote the set of all intervals
$[W_{1},W_{2}]$ such that $\zero \subseteq W_{1} \subseteq W_{2} \subseteq X$.
Let $W\subseteq \F_{q}^{n}$ be a subspace with $\dim W=d$.
We prove \ref{it:Bn(q)lem1} and \ref{it:Bn(q)lem2} by induction on $d$.
\begin{enumerate}[label=(\roman*),leftmargin=2em]
\item\label{it:d=0-Bn(d)}
If $d=0$ then $W=\zero$ is the only subspace of itself.
Hence, by \eqref{eq:FT_of_T_game_1}, 
the NIM-sum ${\nimsum}_{T\subsetneq W}g_{B_{n}(q),\I}(T)$ is empty and therefore equals $0$.
Consequently, $g_{B_{n}(q),\I}(W)=\mex(\{0\})=1$,
This verifies \eqref{eq:FT_of_T_game_1} in the case $d=0$, 
and we conclude that $g_{q}(d)=1$.
\item\label{it:d>0-Bn(d)}
Assume now that $d>0$, and \ref{it:Bn(q)lem1} and \ref{it:Bn(q)lem2} hold for all dimensions less than $d$.
Let $W \subseteq \F_{q}^{n}$ be a subspace with $\dim W=d$.
By Theorem~\ref{thm:FT_of_T_game} \eqref{eq:FT_of_T_game_1},
we have 
\begin{equation}
g_{B_{n}(q),\I}(W)=\mex\Bigl\{
\underset{U\subseteq T\subsetneq W}{\nimsum}g_{B_{n}(q),\I}(T)
\,\Big|\,
U\subseteq W\Bigr\}.
\label{eq:middle-subspaces}
\end{equation}
By the induction hypothesis,
$g_{B_{n}(q),\I}(T)$ depends only on $\dim T=k$ since $k<d$;
hence we may write $g_{B_{n}(q),\I}(T)=g_{q}(k)$.
If we set $\dim U=m$,
then $m$ ranges over $0,1,\dots,d$.
The interval $[U,W]$ is isomorphic, as a vector space poset, to the poset of subspaces of the quotient space $W/U$.
Therefore, the number of subspaces $T\in[U,W]$ with $\dim T=k$ is equal to the number of subspaces $\overline T$ of $W/U$ with $\dim \overline T=k-m$,
which is $\qbinom{d-m}{k-m}{q}$ since $\dim W/U=d-m$
(see \cite{ec1}).
Substituting this into \eqref{eq:middle-subspaces}, we obtain
\[
g_{B_{n}(q),\I}(W)=\mex\Bigl\{ \overset{d-1}{\underset{k=m}{\nimsum}}\qbinom{d-m}{k-m}{q}g_{q}(k)\mid m=0,1,\dots,d\Bigr\}.
\]
Since $(\Non,\nimadd)$ is an abelian group of exponent $2$, we may reduce the coefficients modulo $2$.
This shows that $g_{B_{n}(q),\I}(W)$ depends only on $d$, and that \eqref{eq:rec-mex} holds.
Hence \ref{it:Bn(q)lem1} and \ref{it:Bn(q)lem2} are valid for $\dim W=d$.
\end{enumerate}
Together, \ref{it:d=0-Bn(d)} and \ref{it:d>0-Bn(d)} complete the proof by induction.
\end{demo}
\begin{figure}[htbp]
\begin{tikzcd}
    & & & {4} & & & \\
    {1} \arrow{rrru} & {1} \arrow{rru} & {1} \arrow{ru} & {1} \arrow{u} & {1} \arrow{lu} & {1} \arrow{llu} & {1} \arrow{lllu} \\
    {2} \arrow{u}\arrow{ru}\arrow{rrru} & {2} \arrow{lu}\arrow{ru}\arrow{rrru} & {2} \arrow{lu}\arrow{ru}\arrow{rrru} & {2} \arrow{lllu}\arrow{rru}\arrow{rrru} & {2} \arrow{lllu}\arrow{u}\arrow{rru} & {2} \arrow{lllu}\arrow{llu}\arrow{ru} & {2} \arrow{lllu}\arrow{llu}\arrow{lu} \\
    & & & {1} \arrow{lllu}\arrow{llu}\arrow{lu}\arrow{u}\arrow{ru}\arrow{rru}\arrow{rrru} \\
\end{tikzcd}
\caption{The Grundy values of $B_3(2)$}
\label{fig:B32g-ruler}
\end{figure}
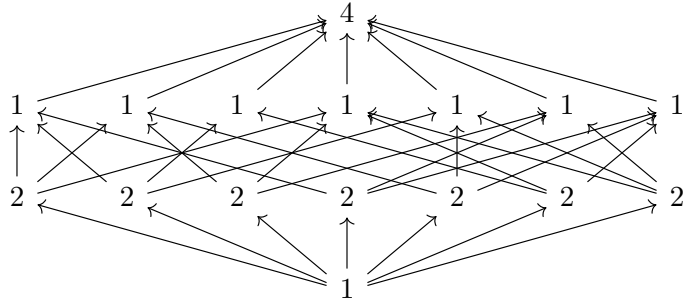
%
%
%
If $q$ is even, then
the recurrence relation \eqref{eq:q-binom-rec} implies $\qbinom{n}{r}{q}\equiv\qbinom{n-1}{r}{q}$ ($\MOD\,2$)
for $r<n$. Hence, we obtain $G_{q}(n,r)=1$ for $0\leq r\leq n$.
\par
If $q$ is odd, 
then from \eqref{eq:q-binom-rec} 
we obtain $G_{q}(n,r)=G_{q}(n-1,r-1) \nimadd G_{q}(n-1,r)$
for $n\geq1$,
since $q^r\equiv 1$ ($\MOD\,2$).
Hence $G_{q}(n,r)$ is independent of $q$.
If we write $G(n,r)$ for $G_{q}(n,r)$,
then
\begin{equation}\label{eq:rec:G(n,r)}
	G(n,r)=G(n-1,r-1) \nimadd G(n-1,r).
\end{equation}
Accordingly,
the NIM-sum in \eqref{eq:rec:sq(d,m)-general} can be written as
\begin{equation}\label{eq:def:s(d,m)}
s_{q}(d,m):=\overset{d-1}{\underset{k=m}{\,\,\,\nimsum}} G(d-m,k-m) g_{q}(k)
\end{equation}
for $0\leq m\leq d$ when $q$ is odd.
Hence \eqref{eq:rec:G(n,r)} together with \eqref{eq:def:s(d,m)} yields the following lemma, 
which serves as a key lemma in the proof of Theorem~\ref{th:Grundy-subsFF-ruler} for odd $q$.
%
%
\begin{lem}\label{lem:red:s(d,m)}
If $d\geq 1$ and $0\leq m\leq d-1$,
then $s_{q}(d,m)$ satisfies 
\begin{equation}\label{eq:rec:s(d,m)}
s_{q}(d,m)=s_{q}(d,m+1)\nimadd s_{q}(d-1,m)\nimadd g_{q}(d-1).
\end{equation}
\end{lem}
\begin{demo}{Proof}
Since $d-m\geq1$,
we may apply the recurrence \eqref{eq:rec:G(n,r)} to obtain
\begin{align*}
s_{q}(d,m)
&=\overset{d-1}{\underset{k=m}{\,\,\,\nimsum}} G(d-m-1,k-m-1) g_{q}(m)
\\&\qquad\qquad
\nimadd\overset{d-1}{\underset{k=m}{\,\,\,\nimsum}} G(d-m-1,k-m) g_{q}(m).
\end{align*}
Note that $G(d-m-1,k-m-1)=0$ in the first sum when $k=m$,
and that $G(d-m-1,k-m)=1$ in the second sum when $k=d-1$.
Therefore, 
\begin{align*}
s_{q}(d,m)
&=\overset{d-1}{\underset{k=m+1}{\,\,\,\nimsum}} G(d-m-1,k-m-1) g_{q}(m)
\\&\qquad\qquad
\nimadd\overset{d-2}{\underset{k=m}{\,\,\,\nimsum}} G(d-m-1,k-m) g_{q}(m)
\nimadd g(d-1)
\\&=s_{q}(d,m+1)\nimadd s_{q}(d-1,m)\nimadd g(d-1),
\end{align*}
which proves \eqref{eq:rec:s(d,m)}.
\end{demo}
\begin{demo}{Proof of Theorem~\ref{th:Grundy-subsFF-ruler}}
\begin{enumerate}[leftmargin=14pt,label=(\roman*)] 
\item\label{item:proof-even}
When $q$ is even,
Lemma~\ref{lem:Bn(q)-recurrence} 
and the identity $G_{q}(n,r)=1$ for $0\leq r\leq n$ 
imply that the Grundy function $g_{q}(d)$ satisfies the initial condition $g_{q}(0)=1$ and, 
for $d>0$, the recurrence
\begin{equation*}
g_{q}(d)=\mex\Bigl\{ \,
\overset{d-1}{\underset{k=m}{\,\,\,\nimsum}} \, g_{q}(k)
\, \Big| \, 
m=0,1,\dots,d
\, \Bigr\}.
\end{equation*}
By Lemma~\ref{lem:chain-ruler-recurrence}, we therefore conclude that $g_{q}(d)=\fnh{d+1}$.
\item\label{item:proof-odd}
\begin{subequations}
Let $q$ be odd, and let $s_{q}(d,m)$ and $g_{q}(d)$ be defined as above.
To prove the theorem we define a function $H:\Int\to\Non$ by $H(x):=\MOD(x,3)+1$.
Then it clearly satisfies
\[
H(x+3)=H(x),\qquad H(x) \nimadd H(x+1)\nimadd H(x+2)=0
\]
for all $x\in\Int$.
To complete the proof for $q$ odd case,
it suffices to prove the following claims.
\begin{enumerate}[label=(\alph*)]
\item\label{it:claim:sq(d,m)}
For nonnegative integers $d,m\in\Non$ with $0\leq m\leq d$,
we have
\begin{equation}\label{eq:claim:sq(d,m)}
s_{q}(d,m)=\begin{cases}
0,\qquad&\text{ if $d \equiv m$ ($\MOD\,3$),}\\
H(m),\qquad &\text{ if $d \not\equiv m$ ($\MOD\,3$).}
\end{cases}
\end{equation}
\item\label{it:claim:gq(d)}
For all $d\in\Non$,
\begin{equation}\label{eq:claim:gq(d)}
g_{q}(d)=H(d).
\end{equation}
\end{enumerate}
\end{subequations}
We prove the claim by induction on $d$,
with a nested induction on $d-m$.
\begin{enumerate}[label=\Roman*)]
\item\label{it:Case:d=0}
If $d=0$, then necessarily $m=0$ since $0\leq m\leq d$.
By definition, the sum in \eqref{eq:def:s(d,m)} is empty, and hence $s_{q}(0,0)=0$.
Moreover, $g_{q}(0)=\mex\{s_{q}(0,0)\}=1$.
Therefore, both claims \ref{it:claim:sq(d,m)} and \ref{it:claim:gq(d)} hold for $d=0$.
\item\label{it:Case:d>0}
Let $d\geq1$,
and assume that both claims \ref{it:claim:sq(d,m)} and \ref{it:claim:gq(d)} 
hold for all smaller values of $d$, that is, for $d-1$.
\begin{enumerate}[label=\roman*)]
\item\label{it:Case:d-m=0}
If $d-m=0$, that is, $m=d$, then the sum in \eqref{eq:def:s(d,m)} is empty, hence $s(d,d)=0$.
This verifies \eqref{eq:claim:sq(d,m)} for $d-m=0$.
\item\label{it:Case:d-m>0}
Now suppose $d-m\geq1$, and assume that \eqref{eq:claim:sq(d,m)} holds for $d-m-1$.
Using the recurrence \eqref{eq:rec:s(d,m)} together with the induction hypothesis,
we compute $s_{q}(d,m)$ as
\begin{equation*}
\begin{cases}
H(m+1)\nimadd H(m)\nimadd H(m-1)=0\quad&\text{ if $d-m \equiv 0$ ($\MOD\,3$)}\\
0\nimadd 0\nimadd H(d-1)=H(m)\quad&\text{ if $d-m \equiv 1$ ($\MOD\,3$)}\\
H(m+1)\nimadd H(m)\nimadd H(d-1)=H(m)\quad&\text{ if $d-m \equiv 2$ ($\MOD\,3$)}
\end{cases}
\end{equation*}
This establishes \eqref{eq:claim:sq(d,m)} for $d-m$.
\end{enumerate}
Combining Cases~\ref{it:Case:d-m=0} and~\ref{it:Case:d-m>0} completes the proof of \eqref{eq:claim:sq(d,m)} for $d$.
\par
Next we prove \eqref{eq:claim:gq(d)} for $d$.
Let $S_{q}(d)=\{s_{q}(d,m)\,\mid\,m=0,1,\dots,d\}$.
Then \eqref{eq:rec-mex} gives $g_{q}(d)=\mex(S_{q}(d))$.
If $d=1$, then we have $S_{q}(d)=\{0,1\}$, and hence $\mex(S_{q}(1))=2=H(1)$, as claimed.
If $d\geq2$, then the above result implies that $H(d)\not\in S_{q}(d)$,
that $0\in S_{q}(d)$, 
and that $H(m)\in S_{q}(d)$ whenever $m\not\equiv d$ ($\MOD\,3$).
Consequently, $\mex(S_{q}(d))=H(d)$, which establishes \eqref{eq:claim:gq(d)} for $d$.
\end{enumerate}
Combining Cases~\ref{it:Case:d=0} and~\ref{it:Case:d>0},
we conclude that both \ref{it:claim:sq(d,m)} and \ref{it:claim:gq(d)} hold for all $d\in\Non$.
\end{enumerate} 
Hence, cases \ref{item:proof-even} and \ref{item:proof-odd} complete 
the proof of the theorem for both the even and odd cases.
\end{demo}
%
\subsection{The Ideal Game on the ASM Poset}\label{sbsec:ASM-ideal}
In this subsection, we focus on the ASM poset $X=\ASM{n}$,
and study the order ideal game on $X$; that is, we set
$
	\T=\OI
	:=
	\left\{
		{\Lambda}_{ \boldsymbol{x} } \in J\left( \ASM{n} \right)
	\,\middle| \,
		\boldsymbol{x} \in \ASM{n}
	\right\}
$
throughout this subsection.
\par
Now,
we denote 
\begin{equation*}
	\Y{n}:=\{
		(r, s) \in \Int^{2}
	\mid
		0 \leq s \leq r \leq n - 2
	\},
\end{equation*}
and define a map
\begin{equation*}
\pi \colon \ASM{n} \rightarrow \Y{n}
,\quad
\left( x, y, z \right) \mapsto \left( n - 2 - (x + y), z \right).
\end{equation*}
We write the coordinate projections of $\pi$ as
$\pi_{1}\left( x, y, z \right)=\rkf(x,y,z)=n - 2 - (x + y)$
and $\pi_{2}\left( x, y, z \right)=z$.
The following lemma then shows that the Grundy value $g\left(\x\right)$ 
depends only on the pair $\left(\pi_{1}\left(\x\right),\pi_{2}\left(\x\right)\right)$ for $\x=(x,y,z)\in \ASM{n}$.
%
\begin{lem}\label{lem:g_value_isomorphic}
Let $n$ be a positive integer.
	In the order ideal game 
	$\mathscr{G}\left( \ASM{n}, \OI \right)$,
	the Grundy value $g\left( \x \right)$ depends only on the pair $\left(\pi_{1}(\x),\pi_{2}(\x)\right)$.
\end{lem}
\begin{demo}{Proof}
Let $\x_{1}=\left( x_{1}, y_{1}, z_{1} \right)$ and $\x_{2}=\left( x_{2}, y_{2}, z_{2} \right)$ be elements of $\ASM{n}$.
If 
$x_{1} + y_{1} = x_{2} + y_{2}$ 
and $z_{1} = z_{2}$,
then the principal order ideals
$\Lambda_{\x_{1}}$ and $\Lambda_{\x_{2}}$ are order isomorphic.
Indeed, setting $l=x_{2} - x_{1}$,
the map
\[
F \colon \Lambda_{\x_{1}} \rightarrow \Lambda_{\x_{2}}
,\qquad
F\left( x, y, z \right)=\left( x + l, y - l, z \right)
\]
is a well-defined order-preserving bijection.
Let $X_{i}=\Lambda_{\x_{i}}$ and $\OI_{i}=\{\Lambda_{\y}\in\OI\mid \y\in X_{i}\}$ for $i=1,2$.
Then the restricted games $\mathscr{G}\left( X_{1}, \OI_{1} \right)$ and $\mathscr{G}\left( X_{2}, \OI_{2} \right)$
satisfy condition \eqref{eq:iso-T} of Proposition~\ref{prop:g_value_order_isomorphism_general}.
Therefore, we have $g(\y)=g\left(F(\y)\right)$ for any $\y\in X_{1}$,
and in particular $g(\x_{1})=g(\x_{2})$.
\end{demo}
\par\smallskip
We next describe explicitly the structure of a principal order ideal in $\ASM{n}$.
%
\begin{lem}\label{lem:POI}
\begin{subequations}
	Let $n$ be a positive integer,
	and let $\x_{0}=(x_{0}, y_{0}, z_{0}) \in \ASM{n}$.
	The principal order ideal ${\Lambda}_{\x_{0}}$ is given by
	\begin{equation}\label{eq:POI}
		{\Lambda}_{\x_{0}}
		=
		\left\{
			\left( x, y, z \right) \in \mathbb{N}^{3}
		\, \middle| \,
			\begin{gathered}
				x \geq x_{0}, \,
				y \geq y_{0}, \,
				z \leq z_{0}, \\
				x_{0} + y_{0} + z_{0} \leq x + y + z \leq n - 2.
			\end{gathered}
		\right\}
		.
	\end{equation}
	For $\x_{0}\in\ASM{n}$ and $(r,s)\in\Y{n}$, let $\Lambda_{\x_{0}}(r,s)$ denote the set of all $\x\in \Lambda_{\x_{0}}$
	satisfying $\pi_{1}(\x)=r$ and $\pi_{2}(\x)=s$.
	Then $\Lambda_{\x_{0}}(r,s)$ is nonempty if and only if $(r,s)\in\R{n}\left(\pi_{1}(\x_{0}),\pi_{2}(\x_{0})\right)$,
where 
\begin{equation}\label{eq:pi-range}
\R{n}(r_{0},s_{0}):=\left\{(r,s)\in\Y{n} \mid 0\leq s\leq s_{0},\, 0\leq r-s\leq r_{0}-s_{0} \right\}.
\end{equation}
Moreover, for each $(r,s)\in\R{n}(\pi_{1}(\x_{0}),\pi_{2}(\x_{0}))$, the cardinality of $\Lambda_{\x_{0}}(r,s)$ is given by
\begin{equation}\label{eq:NumEleIdeal}
\left|\Lambda_{\x_{0}}(r,s)\right|=\pi_{1}(\x_{0})-r+1,
\end{equation}
which is, in particular, independent of $s$.
\end{subequations}
\end{lem}
\begin{demo}{Proof}
The identity \eqref{eq:POI} follows immediately from the definition \eqref{eq:ASM-order}.
Let $\pi(\x_{0})=(r_{0},s_{0})$.
Since
$r_{0}=n-2-x_{0}-y_{0}$,
$s_{0}=z_{0}$,
$r=n-2-x-y$,
and $s=z$,
the inequality
$x_{0} + y_{0} + z_{0} \leq x + y + z \leq n - 2$
is equivalent to $0\leq r-s\leq r_{0}-s_{0}$.
Furthermore, by \eqref{eq:POI} we have $s=z\leq z_{0}=s_{0}$.
Consequently, we obtain 
\[
\max(r-r_{0}+s_{0},0)\leq s\leq \min(r,s_{0}),
\]
which is equivalent to \eqref{eq:pi-range}.
\par
Finally, substituting $y=n-2-r-x$ and $y_{0}=n-2-r_{0}-x_{0}$ into the inequality $y\geq y_{0}$
yields 
$x_{0}\leq x\leq x_{0}+r_{0}-r$.
Hence, for fixed $r$ and $s$, there are exactly $r_{0}-r+1$ possible values of $x$.
This completes the proof of \eqref{eq:NumEleIdeal}.
\end{demo}
%
%
%
\begin{figure}[htbp]
\centering
\begin{subfigure}[b]{0.45\textwidth}
    \centering
	\begin{tikzpicture}[scale=0.75]
		\coordinate (base_x) at (225:1);
		\coordinate (base_y) at (1, 0);
		\coordinate (base_z) at (0, 1);
		\draw[->, help lines, dashed] 
			($0*(base_x) + 0*(base_y) + 0*(base_z)$)
			--
			($4*(base_x) + 0*(base_y) + 0*(base_z)$);
		\node at ($4.25*(base_x) + 0*(base_y) + 0*(base_z)$){$x$};
		\draw[->, help lines, dashed] 
			($0*(base_x) + 0*(base_y) + 0*(base_z)$)
			--
			($0*(base_x) + 4*(base_y) + 0*(base_z)$);
		\node at ($0*(base_x) + 4.25*(base_y) + 0*(base_z)$){$y$};
		\draw[->, help lines, dashed] 
			($0*(base_x) + 0*(base_y) + 0*(base_z)$)
			--
			($0*(base_x) + 0*(base_y) + 4*(base_z)$);
		\node at ($0*(base_x) + 0*(base_y) + 4.25*(base_z)$){$z$};
		\foreach \x in {0, 1, 2, 3}{
			\coordinate (v_\x_0_0) at ($\x*(base_x) + 0*(base_y) + 0*(base_z)$);
		}
		\foreach \x in {0, 1, 2}{
			\coordinate (v_\x_1_0) at ($\x*(base_x) + 1*(base_y) + 0*(base_z)$);
		}
		\foreach \x in {0, 1}{
			\coordinate (v_\x_2_0) at ($\x*(base_x) + 2*(base_y) + 0*(base_z)$);
		}
		\foreach \x in {0}{
			\coordinate (v_\x_3_0) at ($\x*(base_x) + 3*(base_y) + 0*(base_z)$);
		}
		\foreach \x in {0, 1, 2}{
			\coordinate (v_\x_0_1) at ($\x*(base_x) + 0*(base_y) + 1*(base_z)$);
		}
		\foreach \x in {0, 1}{
			\coordinate (v_\x_1_1) at ($\x*(base_x) + 1*(base_y) + 1*(base_z)$);
		}
		\foreach \x in {0}{
			\coordinate (v_\x_2_1) at ($\x*(base_x) + 2*(base_y) + 1*(base_z)$);
		}
		\foreach \x in {0, 1}{
			\coordinate (v_\x_0_2) at ($\x*(base_x) + 0*(base_y) + 2*(base_z)$);
		}
		\foreach \x in {0}{
			\coordinate (v_\x_1_2) at ($\x*(base_x) + 1*(base_y) + 2*(base_z)$);
		}
		\foreach \x in {0}{
			\coordinate (v_\x_0_3) at ($\x*(base_x) + 0*(base_y) + 3*(base_z)$);
		}
		\foreach \x/\xx in {3/2}{
			\draw[directed] (v_\x_0_0) -- (v_\xx_0_0);
		}
		\foreach \x/\xx in {2/1, 1/0}{
			\draw[directed, help lines] (v_\x_0_0) -- (v_\xx_0_0);
		}
		\foreach \x/\xx in {2/1}{
			\draw[directed] (v_\x_1_0) -- (v_\xx_1_0);
		}
		\foreach \x/\xx in {1/0}{
			\draw[directed, help lines] (v_\x_1_0) -- (v_\xx_1_0);
		}
		\foreach \x/\xx in {1/0}{
			\draw[directed] (v_\x_2_0) -- (v_\xx_2_0);
		}
		\foreach \y/\yy in {3/2}{
			\draw[directed] (v_0_\y_0) -- (v_0_\yy_0);
		}
		\foreach \y/\yy in {2/1, 1/0}{
			\draw[directed, help lines] (v_0_\y_0) -- (v_0_\yy_0);
		}
		\foreach \y/\yy in {2/1}{
			\draw[directed] (v_1_\y_0) -- (v_1_\yy_0);
		}
		\foreach \y/\yy in {1/0}{
			\draw[directed, help lines] (v_1_\y_0) -- (v_1_\yy_0);
		}
		\foreach \y/\yy in {1/0}{
			\draw[directed] (v_2_\y_0) -- (v_2_\yy_0);
		}
		\foreach \x/\xx in {2/1}{
			\draw[directed] (v_\x_0_1) -- (v_\xx_0_1);
		}
		\foreach \x/\xx in {1/0}{
			\draw[directed, help lines] (v_\x_0_1) -- (v_\xx_0_1);
		}
		\foreach \x/\xx in {1/0}{
			\draw[directed] (v_\x_1_1) -- (v_\xx_1_1);
		}
		\foreach \y/\yy in {2/1}{
			\draw[directed] (v_0_\y_1) -- (v_0_\yy_1);
		}
		\foreach \y/\yy in {1/0}{
			\draw[directed, help lines] (v_0_\y_1) -- (v_0_\yy_1);
		}
		\foreach \y/\yy in {1/0}{
			\draw[directed] (v_1_\y_1) -- (v_1_\yy_1);
		}
		\foreach \x/\xx in {1/0}{
			\draw[directed] (v_\x_0_2) -- (v_\xx_0_2);
		}
		\foreach \y/\yy in {1/0}{
			\draw[directed] (v_0_\y_2) -- (v_0_\yy_2);
		}
		\foreach \x/\z/\xx/\zz in {3/0/2/1, 2/1/1/2}{
			\draw[directed] (v_\x_0_\z) --(v_\xx_0_\zz);
		}
		\foreach \x/\z/\xx/\zz in {1/2/0/3}{
			\draw[directed, help lines] (v_\x_0_\z) --(v_\xx_0_\zz);
		}
		\foreach \x/\z/\xx/\zz in {2/0/1/1, 1/1/0/2}{
			\draw[directed] (v_\x_1_\z) --(v_\xx_1_\zz);
		}
		\foreach \x/\z/\xx/\zz in {1/0/0/1}{
			\draw[directed] (v_\x_2_\z) --(v_\xx_2_\zz);
		}
		\foreach \y/\z/\yy/\zz in {3/0/2/1, 2/1/1/2}{
			\draw[directed] (v_0_\y_\z) --(v_0_\yy_\zz);
		}
		\foreach \y/\z/\yy/\zz in {1/2/0/3}{
			\draw[directed, help lines] (v_0_\y_\z) --(v_0_\yy_\zz);
		}
		\foreach \y/\z/\yy/\zz in {2/0/1/1, 1/1/0/2}{
			\draw[directed] (v_1_\y_\z) --(v_1_\yy_\zz);
		}
		\foreach \y/\z/\yy/\zz in {1/0/0/1}{
			\draw[directed] (v_2_\y_\z) --(v_2_\yy_\zz);
		}
		\foreach \x/\z/\xx/\zz in {2/0/1/1, 1/1/0/2}{
			\draw[directed] (v_\x_0_\z) --(v_\xx_0_\zz);
		}
		\foreach \x/\z/\xx/\zz in {1/0/0/1}{
			\draw[directed] (v_\x_1_\z) --(v_\xx_1_\zz);
		}
		\foreach \y/\z/\yy/\zz in {2/0/1/1, 1/1/0/2}{
			\draw[directed] (v_0_\y_\z) --(v_0_\yy_\zz);
		}
		\foreach \y/\z/\yy/\zz in {1/0/0/1}{
			\draw[directed] (v_1_\y_\z) --(v_1_\yy_\zz);
		}
		\foreach \x/\z/\xx/\zz in {1/0/0/1}{
			\draw[directed, help lines] (v_\x_0_\z) --(v_\xx_0_\zz);
		}
		\foreach \x in {0, 1}{
			\fill[draw=black, fill=white] (v_\x_0_0) circle (2pt);
		}
		\foreach \x in {2, 3}{
			\fill[draw=black, fill=white] (v_\x_0_0) circle (2pt);
			\draw[fill=black, fill opacity=0.3] (v_\x_0_0) circle [radius=2pt];
		}
		\foreach \x in {0}{
			\fill[draw=black, fill=white] (v_\x_1_0) circle (2pt);
		}
		\foreach \x in {1, 2}{
			\fill[draw=black, fill=white] (v_\x_1_0) circle (2pt);
			\draw[fill=black, fill opacity=0.3] (v_\x_1_0) circle [radius=2pt];
		}
		\foreach \x in {0, 1}{
			\fill[draw=black, fill=white] (v_\x_2_0) circle (2pt);
			\draw[fill=black, fill opacity=0.3] (v_\x_2_0) circle [radius=2pt];
		}
		\foreach \x in {0}{
			\fill[draw=black, fill=white] (v_\x_3_0) circle (2pt);
			\fill[draw=black, fill=black, fill opacity=0.3] (v_\x_3_0) circle (2pt);
		}
		\foreach \x in {0}{
			\fill[draw=black, fill=white] (v_\x_0_1) circle (2pt);
		}
		\foreach \x in {1, 2}{
			\fill[draw=black, fill=white] (v_\x_0_1) circle (2pt);
			\fill[draw=black, fill=black, fill opacity=0.3] (v_\x_0_1) circle (2pt);
		}
		\foreach \x in {0, 1}{
			\fill[draw=black, fill=white] (v_\x_1_1) circle (2pt);
			\fill[draw=black, fill=black, fill opacity=0.3] (v_\x_1_1) circle (2pt);
		}
		\foreach \x in {0}{
			\fill[draw=black, fill=white] (v_\x_2_1) circle (2pt);
			\fill[draw=black, fill=black, fill opacity=0.3] (v_\x_2_1) circle (2pt);
		}
		\foreach \x in {0}{
			\fill[draw=black, fill=black] (v_\x_0_2) circle (2pt) node [above right] {\tiny$(0,0,2)$};
		}
		\foreach \x in {1}{
			\fill[draw=black, fill=white] (v_\x_0_2) circle (2pt);
			\fill[draw=black, fill=black, fill opacity=0.3] (v_\x_0_2) circle (2pt);
		}
		\foreach \x in {0}{
			\fill[draw=black, fill=white] (v_\x_1_2) circle (2pt);
			\fill[draw=black, fill=black, fill opacity=0.3] (v_\x_1_2) circle (2pt);
		}
		\foreach \x in {0}{
			\fill[draw=black, fill=white] (v_\x_0_3) circle (2pt);
		}
	\end{tikzpicture}
	\vskip-1cm
    \caption{\scriptsize The elements of ${\Lambda}_{ (0, 0, 2) }$}
    \label{fig:ideal002a}
\end{subfigure}
\hfill
\begin{subfigure}[b]{0.45\textwidth}
    \centering
\begin{tikzpicture}[scale=0.75]
\pgfmathsetmacro{\m}{3}
\pgfmathsetmacro{\l}{int((\m-1)/2)}
\pgfmathsetmacro{\mm}{\m-1}
    \draw [
        pattern=dots, 
        pattern color=black!50!white,
        draw=black,
        line width=0pt
    ] (0, 0) -- (1, 0) -- (3, 2) -- (2, 2) -- cycle;
\foreach \y in {1,...,\mm}{
  \draw [help lines] (\y,\y)--(\m,\y);
}
\draw [>-Stealth,help lines] (0,0)--(\m+0.8,0) node [right] {$r=$rank};
\foreach \x in {1,...,\mm}{
  \draw [help lines] (\x,0)--(\x,\x);
}
\draw [>-Stealth,help lines] (\m,0)--(\m,\m+0.8) node [above] {$s$};
\foreach \x in {0,...,\m} {
  \pgfmathsetmacro{\px}{7-\x}
  \foreach \y in {0,...,\x} {
    \fill[draw=black,fill=white] (\x,\y) circle [radius=3pt];
  }
}
\foreach \x in {0,...,\m} {
  \draw (\x,0) circle [radius=3pt] node[text opacity=0.5,below left] {\scriptsize$\x$};
}
\foreach \y in {0,...,\m} {
  \draw (\m,\y) circle [radius=3pt] node [text opacity=0.5,above right] {\scriptsize$\y$};
}
\coordinate (O) at (3,2);
\fill [black] (O) circle [radius=3pt];
\foreach \x in {0,...,2}{
  \draw[fill=black, fill opacity=0.3] (\x+1,\x) circle [radius=3pt];
}
\foreach \x in {0,...,2}{
  \draw[fill=black, fill opacity=0.3] (\x,\x) circle [radius=3pt];
}
\draw (0,0) circle [radius=3pt] node [text opacity=0.8,above left] {\scriptsize$4$};
\draw (1,0) circle [radius=3pt] node [text opacity=0.8,above left] {\scriptsize$3$};
\draw (1,1) circle [radius=3pt] node [text opacity=0.8,above left] {\scriptsize$3$};
\draw (2,0) circle [radius=3pt] node [text opacity=0.8,above left] {\scriptsize$0$};
\draw (2,1) circle [radius=3pt] node [text opacity=0.8,above left] {\scriptsize$2$};
\draw (2,2) circle [radius=3pt] node [text opacity=0.8,above left] {\scriptsize$2$};
\draw (3,0) circle [radius=3pt] node [text opacity=0.8,above left] {\scriptsize$0$};
\draw (3,1) circle [radius=3pt] node [text opacity=0.8,above left] {\scriptsize$0$};
\draw (3,2) circle [radius=3pt] node [text opacity=0.8,above left] {\scriptsize$1$};
\draw (3,3) circle [radius=3pt] node [text opacity=0.8,above left] {\scriptsize$0$};
\end{tikzpicture}
    \caption{$\left|\Lambda_{(0,0,2)}(r,s)\right|$}
    \label{fig:ideal002b}
\end{subfigure}
\caption{${\Lambda}_{ (0, 0, 2) } \in J\left( \ASM{5} \right)$}
\label{fig:ideal002}
\end{figure}
%
%
\par
We illustrate the order ideal ${\Lambda}_{(0, 0, 2)} \in J\left( \ASM{5} \right)$
in Figure~\ref{fig:ideal002a}.
In this figure,
$(0, 0, 2)$ is drawn as a closed circle;
each element in $\Lambda_{(0, 0, 2)}\setminus\{(0,0,2)\}$ is drawn as a grey circle;
and all other elements of $\ASM{5}$ are drawn as open circles.
\par
Figure~\ref{fig:ideal002b} illustrates the image of $\Lambda_{(0, 0, 2)}$
under $\pi$.
The closed circle represents $\pi\left(0, 0, 2\right)$,
and the grey circles represent the points in $\pi\left(\Lambda_{(0, 0, 2)}\setminus\{(0, 0, 2)\}\right)$;
whereas the open circles lie outside this image.
For each vertex,
the number written to its upper left indicates how many elements of $\Lambda_{(0, 0, 2)}$
are mapped by $\pi$ to the coordinates $(r,s)$.
\par
\begin{demo}{Proof of Theorem~\ref{th:ASM-ideal}}
The Grundy function takes only the values $0$ or $1$,
according to the condition \eqref{eq:OIgame01} in Lemma~\ref{lem:OrderI01}.
Hence, for each $\x_{0}\in\ASM{n}$, we determine $g(\x_{0})$
by counting the number of elements in
$\Lambda_{\x_{0}}\setminus\{\x_{0}\}$ whose Grundy values is $1$.
\par
The identity \eqref{eq:ASM-ideal-grundy} appearing in Theorem~\ref{th:ASM-ideal} can be equivalently restated as follows.
\par\smallskip\noindent
\cqs\hspace*{2ex}
\begin{minipage}[]{\textwidth-\cqswidth-2ex}
The elements whose Grundy value is $1$ are precisely those $(x,y,z)$ satisfying:
\begin{enumerate}[label=(\Alph*)]
\item\label{it:(0,0)}
$\pi(x,y,z)=(0,0)$, or
\item\label{it:odd-rank}
$\pi(x,y,z)=(2k+1,k)$ or $(2k+1,k+1)$ for $k\in\Non$.
\end{enumerate}
\end{minipage}
\par\smallskip\noindent
Hereafter, we write $P_{k}:=(2k+1,k)$ and $Q_{k}:=(2k+1,k+1)$ for each $k\in\Non$.
Figure~\ref{fig:pi(ideal)} illustrates this condition:
the closed circles represent the points whose Grundy value is $1$,
whereas the open circles represent those whose Grundy value is $0$.
\par
Assume $\x_{0}\in\ASM{n}$ and $\pi(\x_{0})=(r_{0},s_{0})$.
We prove $(*)$ by induction on $\rkf(\x_{0})=r_{0}$.
\par\noindent
(i) If $\rkf(\x_{0})=0$, then $\x_{0}$ is a minimal element of $\ASM{n}$.
Hence $\Lambda_{\x_{0}}\setminus\{\x_{0}\}=\emptyset$, which implies
$\underset{t\in\Lambda_{\x_{0}}\setminus\{\x_{0}\}}{\nimsum}g(t)=0$.
Therefore, by \eqref{eq:FT_of_T_game_1}, we obtain $g(\x_{0})=\mex\{0\}=1$.
This shows that $\x_{0}$ is a type \ref{it:(0,0)} element.
\par\noindent
(ii) Assume $\rkf(\x_{0})>0$.
If $\x\in\Lambda_{\x_{0}}\setminus\{\x_{0}\}$ then $\x<\x_{0}$, and hence $\rkf(\x)<\rkf(\x_{0})$.
Therefore, $(*)$ holds for every $\x\in\Lambda_{\x_{0}}\setminus\{\x_{0}\}$ by the induction hypothesis.
\begin{enumerate}[label=\arabic*), leftmargin=20pt]
\item
First, assume that $\rkf(\x_{0})=r_{0}$ is odd;
equivalently, there exists $k_{0}\in\Non$ such that $r_{0}=2k_{0}+1$.
We begin by counting the elements of type \ref{it:(0,0)} in $\Lambda_{\x_{0}}\setminus\{\x_{0}\}$.
By \eqref{eq:NumEleIdeal},
the number of elements $\x\in\Lambda_{\x_{0}}\setminus\{\x_{0}\}$ with $\pi(\x)=(0,0)$ is $2k_{0}+2$,
which is even,
and thus these elements do not affect the parity of the cardinality of the elements in $\Lambda_{\x_{0}}\setminus\{\x_{0}\}$ 
whose Grundy value equals $1$.
\par
Next we count the elements of type \ref{it:odd-rank} in $\Lambda_{\x_{0}}\setminus\{\x_{0}\}$.
Observe that for each $k,s\in\Non$ such that $(2k+1,s)\in\R{n}(r_{0},s_{0})$,
the number of elements $\x\in\Lambda_{\x_{0}}$ satisfying $\pi(\x)=(2k+1,s)$ is $2(k_{0}-k)+1$,
which is odd.
We now check the following two cases.
\begin{enumerate}[label=\alph*)]
\item
Suppose $0\leq s_{0}\leq k_{0}$.
Then it is straightforward to see that $P_{k}\in \R{n}(r_{0},s_{0})$ for $0\leq k\leq s_{0}$,
wheras $Q_{k}\in \R{n}(r_{0},s_{0})$ for $0\leq k< s_{0}$.
For all other values of $k$, neither $P_{k}$ nor $Q_{k}$ lies in $\R{n}(r_{0},s_{0})$.
If $s_{0}< k_{0}$,
then neither $P_{k}$ nor $Q_{k}$ can coincide with $\pi(\x_{0})$,
and hence the total number of elements of type~\ref{it:odd-rank} is odd.
Therefore $g(\x_{0})=0$ by Lemma~\ref{lem:OrderI01}.
If $s_{0}=k_{0}$,
then an element $\x$ satisfing $\pi(\x)=P_{k_{0}}$ can occur only in the case $\x=\x_{0}$.
Excluding this case,
the total number of elements of type \ref{it:odd-rank} in $\Lambda_{\x_{0}}\setminus\{\x_{0}\}$
is even.
Therefore $g(\x_{0})=1$ by Lemma~\ref{lem:OrderI01}.
\item
Suppose $s_{0}\geq k_{0}+1$.
Then it is also easy to see that $P_{k}\in \R{n}(r_{0},s_{0})$ if $0\leq k\leq 2k_{0}-s_{0}$,
wheras $Q_{k}\in \R{n}(r_{0},s_{0})$ if $0\leq k\leq 2k_{0}-s_{0}+1$.
For all other values of $k$, neither $P_{k}$ nor $Q_{k}$ lies in $\R{n}(r_{0},s_{0})$.
If $s_{0}=k_{0}+1$,
then $\pi(\x)=Q_{k_{0}}$ can occur only in the case $\x=\x_{0}$.
Excluding this case,
the total number of elements of type \ref{it:odd-rank} in $\Lambda_{\x_{0}}\setminus\{\x_{0}\}$
is even,
and therefore $g(\x_{0})=1$ by Lemma~\ref{lem:OrderI01}.
If $s_{0}> k_{0}+1$,
then neither $P_{k}$ nor $Q_{k}$ can coincide with $\pi(\x_{0})$.
Consequently, the total number of elements of type \ref{it:odd-rank} is odd,
and therefore $g(\x_{0})=0$ by Lemma~\ref{lem:OrderI01}.
\end{enumerate}
This completes the proof of $(*)$ when $\rkf(\x_{0})=r_{0}=2k_{0}+1$.
\item
Next, assume that $\rkf(\x_{0})=r_{0}$ is even.
Then there exists $k_{0}\in\Non$ such that $r_{0}=2k_{0}$.
We first count the elements of type~\ref{it:(0,0)}.
By \eqref{eq:NumEleIdeal},
the number of elements $\x\in\Lambda_{\x_{0}}\setminus\{\x_{0}\}$ with $\pi(\x)=(0,0)$ is $2k_{0}+1$,
which is odd.
\par\smallskip
Next we count the elements of type~\ref{it:odd-rank} in $\Lambda_{\x_{0}}\setminus\{\x_{0}\}$.
By \eqref{eq:NumEleIdeal},
the number of elements $\x\in\Lambda_{\x_{0}}\setminus\{\x_{0}\}$ with $\rkf(\x)=2k+1$ is $2(k_{0}-k)$,
which is even.
Thus these elements do not contribute to the parity.
Consequently, the total number of elements in $\Lambda_{\x_{0}}\setminus\{\x_{0}\}$
whose Grundy value is $1$ is odd.
Therefore, by Lemma~\ref{lem:OrderI01}, we conclude that $g(\x_{0})=0$.
This proves $(*)$ when $\rkf(\x_{0})=r_{0}=2k_{0}$.
\end{enumerate}
Consequently we conclude that $(*)$ holds for all $\x_{0}$ with $\rkf(\x_{0})=r_{0}$.
By induction, the proof is complete.
\end{demo}
%
%
%
%
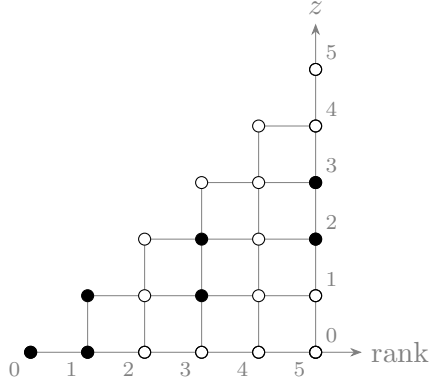
\begin{figure}[htbp]
\centering
\begin{tikzpicture}[scale=0.75]
\pgfmathsetmacro{\m}{5}
\pgfmathsetmacro{\l}{int((\m-1)/2)}
\pgfmathsetmacro{\mm}{\m-1}
\foreach \y in {1,...,\mm}{
  \draw [help lines] (\y,\y)--(\m,\y);
}
\draw [>-Stealth,help lines] (0,0)--(\m+0.8,0) node [right] {rank};
\foreach \x in {1,...,\mm}{
  \draw [help lines] (\x,0)--(\x,\x);
}
\draw [>-Stealth,help lines] (\m,0)--(\m,\m+0.8) node [above] {$z$};
\foreach \x in {0,...,\m} {
  \pgfmathsetmacro{\px}{7-\x}
  \foreach \y in {0,...,\x} {
    \fill[draw=black,fill=white] (\x,\y) circle [radius=3pt];
  }
}
\foreach \x in {0,...,\m} {
  \draw (\x,0) circle [radius=3pt] node[text opacity=0.5,below left] {\scriptsize$\x$};
}
\foreach \y in {0,...,\m} {
  \draw (\m,\y) circle [radius=3pt] node [text opacity=0.5,above right] {\scriptsize$\y$};
}
\coordinate (O) at (0,0);
\fill [black] (O) circle [radius=3pt];
\foreach \x in {0,...,\l}{
  \fill[black] (2*\x+1,\x+0) circle [radius=3pt];
  \fill[black] (2*\x+1,\x+1) circle [radius=3pt];
}
\end{tikzpicture}
\caption{Place of $1$'s in $\R{n}(r,s)$}
\label{fig:pi(ideal)}
\end{figure}
%
%

%
%
%
%
%
\section{Open problems}\label{sc:open}
%
%
%
%
\subsection{The Ruler on the ASM Poset}
We assume that the underlying poset is $X=\ASM{n}$ throughout this subsection.
The order ideal game on $X$ is studied in \S~\ref{sbsec:ASM-ideal}; 
in the present subsection we instead take the turning set to be the family $\I$ of all intervals in $X$
and consider the ruler on $X$.
\par
In Lemma~\ref{lem:g_value_isomorphic},
we show that $g_{\OI}(\x_{1})=g_{\OI}(\x_{2})$ whenever $\pi(\x_{1})=\pi(\x_{2})$.
The same property holds for the ruler on $\ASM{n}$; that is,
$g_{\I}(\x_{1})=g_{\I}(\x_{2})$ whenever $\pi(\x_{1})=\pi(\x_{2})$,
since the interval family $\I$ satisfies condition \eqref{eq:iso-T}.
Consequently, the map ${\overline g}_{\I}\colon \Y{n} \to \Non$, ${\overline g}_{\I}(s,t)=g_{\I}(\x)$ is well-defined,
where $\x \in \ASM{n}$ is any element satisfying $\pi(\x)=(s,t)$.
\par
Next, we present further order isomorphisms that lead to additional identifications of Grundy values.
%
\begin{prop}\label{lem:isomorphism_eta}
	Let $n$ be a positive integer,
	and let $\x = \left( x, y, z \right) \in \ASM{n}$.
	Define maps $\xi,\eta \colon \ASM{n} \to \ASM{n}$ by
	\[
	\xi\left( \x \right)=\left( y, x, z \right)
	\text{ and }
	\eta\left( \x \right)=\left( x, y, n - 2 - (x + y + z) \right).
	\]
	Then both $\xi$ and $\eta$ are order-preserving bijections of $\ASM{n}$.
	Further, for each of the families of turning sets $\T=\TT$, $\OI$, $\I$,
	the maps $\xi$ and $\eta$ satisfy condition \eqref{eq:iso-T} of Proposition~\ref{prop:g_value_order_isomorphism_general}.
	Consequently,
	we have $g_{\I}(x)=g_{\I}(\xi(x))=g_{\I}(\eta(x))$ for all $x\in\ASM{n}$.
	In particular, the involution $\eta$ yeilds the symmetry ${\overline g}_{\I}(s,t)={\overline g}_{\I}(s,s-t)$ for all $(s,t)\in \Y{n}$.
\end{prop}
\begin{demo}{Proof}
Both maps $\xi$ and $\eta$ are involutions and hence bijective.
A straightforward verification shows that they preserve the partial order on $\ASM{n}$, 
and we leave the details to the reader.
We now prove the final claim.
Let $(s,t)\in \Y{n}$, and choose $\x=(x,y,z)\in \ASM{n}$ such that $\pi(\x)=(n-2-(x+y),z)=(s,t)$.
Applying the involution $\eta$, we obtain $\pi\left(\eta(\x)\right)=\left(n-2-(x+y),n-2-(x+y+z)\right)=(s,s-t)$.
Since the Grundy function $g_{\I}$ is invariant under $\eta$, it follows that ${\overline g}_{\I}(s,t)={\overline g}_{\I}(s,s-t)$.
\end{demo}
\par\smallskip
Table~\ref{table:G_value-ruler} lists the values of the Grundy function $\overline g_{\I}(s,t)$ for small pairs $(s,t)$.
At present, no apparent closed-form description or simple pattern for these values is known.
\begin{table}[htbp]
	\begin{center}
	\begin{tikzpicture}[scale=0.5]
		\coordinate (base_1) at 
			(1, 0);
		\coordinate (base_2) at 
			(0, 1);
		\draw[->] 
			($- 0.5*(base_1) + 0.5*(base_2)$)
			--
			($- 0.5*(base_1) + 15.5*(base_2)$);
		\node at ($- 0.5*(base_1) + 15.5*(base_2)$)
			[above]
			{$z$};
		\foreach \z in {0, 1, ..., 14}{
			\node at ($- 0.5*(base_1) + \z*(base_2) + 0.5*(base_2)$)
				[left]
				{$\z$};
			\draw[help lines, dashed]
				 ($- 0.5*(base_1) + \z*(base_2) + 0.5*(base_2)$)
				 --
				 ($\z*(base_2) + 0.5*(base_2) + \z*(base_1)$);
		}
		\draw[->] 
			($0.5*(base_1) - 0.5*(base_2)$)
			--
			($15.5*(base_1) - 0.5*(base_2)$);
		\node at ($15.5*(base_1) - 0.5*(base_2)$)
			[right]
			{rank};
		\foreach \r in {0, 1, ..., 14}{
			\node at ($- 0.5*(base_2) + \r*(base_1) + 0.5*(base_1)$)
				[below]
				{$\r$};
			\draw[help lines, dashed]
				($- 0.5*(base_2) + \r*(base_1) + 0.5*(base_1)$)
				--
				($\r*(base_1) + 0.5*(base_1)$);
		}
		\coordinate (v_0) at ($0.5*(base_1) + 0.5*(base_2)$);
		\fill[draw=black, fill=gray, opacity=0.5]
			($(v_0) + 0*(base_1) - 0.5*(base_1) - 0.5*(base_2)$)
			--
			($(v_0) + 0*(base_1) + 0.5*(base_1) - 0.5*(base_2)$)
			--
			($(v_0) + 0*(base_1) + 0.5*(base_1) + 0.5*(base_2)$)
			--
			($(v_0) + 0*(base_1) - 0.5*(base_1) + 0.5*(base_2)$)
			--
			cycle;
		\foreach \i in {1, 2, ..., 14}{
			\fill[draw=black, fill=gray, opacity=0.5]
				($(v_0) + \i*(base_1) - 0.5*(base_1) - 0.5*(base_2)$)
				--
				($(v_0) + \i*(base_1) + 0.5*(base_1) - 0.5*(base_2)$)
				--
				($(v_0) + \i*(base_1) + 0.5*(base_1) + 0.5*(base_2)$)
				--
				($(v_0) + \i*(base_1) - 0.5*(base_1) + 0.5*(base_2)$)
				--
				cycle;
			\fill[draw=black, fill=gray, opacity=0.5]
				($(v_0) + \i*(base_1) + \i*(base_2) - 0.5*(base_1) - 0.5*(base_2)$)
				--
				($(v_0) + \i*(base_1) + \i*(base_2) + 0.5*(base_1) - 0.5*(base_2)$)
				--
				($(v_0) + \i*(base_1) + \i*(base_2) + 0.5*(base_1) + 0.5*(base_2)$)
				--
				($(v_0) + \i*(base_1) + \i*(base_2) - 0.5*(base_1) + 0.5*(base_2)$)
				--
				cycle;
		}
		\foreach \z in {1}{
			\coordinate (v_\z) at ($\z*(base_2) + 0.5*(base_2) + 2*\z*(base_1) + 0.5*(base_1)$);
			\fill[draw=black, fill=white, opacity=0.5]
				($(v_\z) + 0*(base_1) - 0.5*(base_1) - 0.5*(base_2)$)
				--
				($(v_\z) + 0*(base_1) + 0.5*(base_1) - 0.5*(base_2)$)
				--
				($(v_\z) + 0*(base_1) + 0.5*(base_1) + 0.5*(base_2)$)
				--
				($(v_\z) + 0*(base_1) - 0.5*(base_1) + 0.5*(base_2)$)
				--
				cycle;
			\foreach \i in {1, 2, ..., 12}{
				\fill[draw=black, fill=white, opacity=0.5]
					($(v_\z) + \i*(base_1) - 0.5*(base_1) - 0.5*(base_2)$)
					--
					($(v_\z) + \i*(base_1) + 0.5*(base_1) - 0.5*(base_2)$)
					--
					($(v_\z) + \i*(base_1) + 0.5*(base_1) + 0.5*(base_2)$)
					--
					($(v_\z) + \i*(base_1) - 0.5*(base_1) + 0.5*(base_2)$)
					--
					cycle;
				\fill[draw=black, fill=white, opacity=0.5]
					($(v_\z) + \i*(base_1) + \i*(base_2) - 0.5*(base_1) - 0.5*(base_2)$)
					--
					($(v_\z) + \i*(base_1) + \i*(base_2) + 0.5*(base_1) - 0.5*(base_2)$)
					--
					($(v_\z) + \i*(base_1) + \i*(base_2) + 0.5*(base_1) + 0.5*(base_2)$)
					--
					($(v_\z) + \i*(base_1) + \i*(base_2) - 0.5*(base_1) + 0.5*(base_2)$)
					--
					cycle;
			}
		}
		\foreach \z in {2}{
			\coordinate (v_\z) at ($\z*(base_2) + 0.5*(base_2) + 2*\z*(base_1) + 0.5*(base_1)$);
			\fill[draw=black, fill=gray, opacity=0.5]
				($(v_\z) + 0*(base_1) - 0.5*(base_1) - 0.5*(base_2)$)
				--
				($(v_\z) + 0*(base_1) + 0.5*(base_1) - 0.5*(base_2)$)
				--
				($(v_\z) + 0*(base_1) + 0.5*(base_1) + 0.5*(base_2)$)
				--
				($(v_\z) + 0*(base_1) - 0.5*(base_1) + 0.5*(base_2)$)
				--
				cycle;
			\foreach \i in {1, 2, ..., 10}{
				\fill[draw=black, fill=gray, opacity=0.5]
					($(v_\z) + \i*(base_1) - 0.5*(base_1) - 0.5*(base_2)$)
					--
					($(v_\z) + \i*(base_1) + 0.5*(base_1) - 0.5*(base_2)$)
					--
					($(v_\z) + \i*(base_1) + 0.5*(base_1) + 0.5*(base_2)$)
					--
					($(v_\z) + \i*(base_1) - 0.5*(base_1) + 0.5*(base_2)$)
					--
					cycle;
				\fill[draw=black, fill=gray, opacity=0.5]
					($(v_\z) + \i*(base_1) + \i*(base_2) - 0.5*(base_1) - 0.5*(base_2)$)
					--
					($(v_\z) + \i*(base_1) + \i*(base_2) + 0.5*(base_1) - 0.5*(base_2)$)
					--
					($(v_\z) + \i*(base_1) + \i*(base_2) + 0.5*(base_1) + 0.5*(base_2)$)
					--
					($(v_\z) + \i*(base_1) + \i*(base_2) - 0.5*(base_1) + 0.5*(base_2)$)
					--
					cycle;
			}
		}
		\foreach \z in {3}{
			\coordinate (v_\z) at ($\z*(base_2) + 0.5*(base_2) + 2*\z*(base_1) + 0.5*(base_1)$);
			\fill[draw=black, fill=white, opacity=0.5]
				($(v_\z) + 0*(base_1) - 0.5*(base_1) - 0.5*(base_2)$)
				--
				($(v_\z) + 0*(base_1) + 0.5*(base_1) - 0.5*(base_2)$)
				--
				($(v_\z) + 0*(base_1) + 0.5*(base_1) + 0.5*(base_2)$)
				--
				($(v_\z) + 0*(base_1) - 0.5*(base_1) + 0.5*(base_2)$)
				--
				cycle;
			\foreach \i in {1, 2, ..., 8}{
				\fill[draw=black, fill=white, opacity=0.5]
					($(v_\z) + \i*(base_1) - 0.5*(base_1) - 0.5*(base_2)$)
					--
					($(v_\z) + \i*(base_1) + 0.5*(base_1) - 0.5*(base_2)$)
					--
					($(v_\z) + \i*(base_1) + 0.5*(base_1) + 0.5*(base_2)$)
					--
					($(v_\z) + \i*(base_1) - 0.5*(base_1) + 0.5*(base_2)$)
					--
					cycle;
				\fill[draw=black, fill=white, opacity=0.5]
					($(v_\z) + \i*(base_1) + \i*(base_2) - 0.5*(base_1) - 0.5*(base_2)$)
					--
					($(v_\z) + \i*(base_1) + \i*(base_2) + 0.5*(base_1) - 0.5*(base_2)$)
					--
					($(v_\z) + \i*(base_1) + \i*(base_2) + 0.5*(base_1) + 0.5*(base_2)$)
					--
					($(v_\z) + \i*(base_1) + \i*(base_2) - 0.5*(base_1) + 0.5*(base_2)$)
					--
					cycle;
			}
		}
		\foreach \z in {4}{
			\coordinate (v_\z) at ($\z*(base_2) + 0.5*(base_2) + 2*\z*(base_1) + 0.5*(base_1)$);
			\fill[draw=black, fill=gray, opacity=0.5]
				($(v_\z) + 0*(base_1) - 0.5*(base_1) - 0.5*(base_2)$)
				--
				($(v_\z) + 0*(base_1) + 0.5*(base_1) - 0.5*(base_2)$)
				--
				($(v_\z) + 0*(base_1) + 0.5*(base_1) + 0.5*(base_2)$)
				--
				($(v_\z) + 0*(base_1) - 0.5*(base_1) + 0.5*(base_2)$)
				--
				cycle;
			\foreach \i in {1,2,3,4,5,6}{
				\fill[draw=black, fill=gray, opacity=0.5]
					($(v_\z) + \i*(base_1) - 0.5*(base_1) - 0.5*(base_2)$)
					--
					($(v_\z) + \i*(base_1) + 0.5*(base_1) - 0.5*(base_2)$)
					--
					($(v_\z) + \i*(base_1) + 0.5*(base_1) + 0.5*(base_2)$)
					--
					($(v_\z) + \i*(base_1) - 0.5*(base_1) + 0.5*(base_2)$)
					--
					cycle;
				\fill[draw=black, fill=gray, opacity=0.5]
					($(v_\z) + \i*(base_1) + \i*(base_2) - 0.5*(base_1) - 0.5*(base_2)$)
					--
					($(v_\z) + \i*(base_1) + \i*(base_2) + 0.5*(base_1) - 0.5*(base_2)$)
					--
					($(v_\z) + \i*(base_1) + \i*(base_2) + 0.5*(base_1) + 0.5*(base_2)$)
					--
					($(v_\z) + \i*(base_1) + \i*(base_2) - 0.5*(base_1) + 0.5*(base_2)$)
					--
					cycle;
			}
		}
		\foreach \z in {5}{
			\coordinate (v_\z) at ($\z*(base_2) + 0.5*(base_2) + 2*\z*(base_1) + 0.5*(base_1)$);
			\fill[draw=black, fill=white, opacity=0.5]
				($(v_\z) + 0*(base_1) - 0.5*(base_1) - 0.5*(base_2)$)
				--
				($(v_\z) + 0*(base_1) + 0.5*(base_1) - 0.5*(base_2)$)
				--
				($(v_\z) + 0*(base_1) + 0.5*(base_1) + 0.5*(base_2)$)
				--
				($(v_\z) + 0*(base_1) - 0.5*(base_1) + 0.5*(base_2)$)
				--
				cycle;
			\foreach \i in {1, 2, ..., 4}{
				\fill[draw=black, fill=white, opacity=0.5]
					($(v_\z) + \i*(base_1) - 0.5*(base_1) - 0.5*(base_2)$)
					--
					($(v_\z) + \i*(base_1) + 0.5*(base_1) - 0.5*(base_2)$)
					--
					($(v_\z) + \i*(base_1) + 0.5*(base_1) + 0.5*(base_2)$)
					--
					($(v_\z) + \i*(base_1) - 0.5*(base_1) + 0.5*(base_2)$)
					--
					cycle;
				\fill[draw=black, fill=white, opacity=0.5]
					($(v_\z) + \i*(base_1) + \i*(base_2) - 0.5*(base_1) - 0.5*(base_2)$)
					--
					($(v_\z) + \i*(base_1) + \i*(base_2) + 0.5*(base_1) - 0.5*(base_2)$)
					--
					($(v_\z) + \i*(base_1) + \i*(base_2) + 0.5*(base_1) + 0.5*(base_2)$)
					--
					($(v_\z) + \i*(base_1) + \i*(base_2) - 0.5*(base_1) + 0.5*(base_2)$)
					--
					cycle;
			}
		}
		\foreach \z in {6}{
			\coordinate (v_\z) at ($\z*(base_2) + 0.5*(base_2) + 2*\z*(base_1) + 0.5*(base_1)$);
			\fill[draw=black, fill=gray, opacity=0.5]
				($(v_\z) + 0*(base_1) - 0.5*(base_1) - 0.5*(base_2)$)
				--
				($(v_\z) + 0*(base_1) + 0.5*(base_1) - 0.5*(base_2)$)
				--
				($(v_\z) + 0*(base_1) + 0.5*(base_1) + 0.5*(base_2)$)
				--
				($(v_\z) + 0*(base_1) - 0.5*(base_1) + 0.5*(base_2)$)
				--
				cycle;
			\foreach \i in {1,2}{
				\fill[draw=black, fill=gray, opacity=0.5]
					($(v_\z) + \i*(base_1) - 0.5*(base_1) - 0.5*(base_2)$)
					--
					($(v_\z) + \i*(base_1) + 0.5*(base_1) - 0.5*(base_2)$)
					--
					($(v_\z) + \i*(base_1) + 0.5*(base_1) + 0.5*(base_2)$)
					--
					($(v_\z) + \i*(base_1) - 0.5*(base_1) + 0.5*(base_2)$)
					--
					cycle;
				\fill[draw=black, fill=gray, opacity=0.5]
					($(v_\z) + \i*(base_1) + \i*(base_2) - 0.5*(base_1) - 0.5*(base_2)$)
					--
					($(v_\z) + \i*(base_1) + \i*(base_2) + 0.5*(base_1) - 0.5*(base_2)$)
					--
					($(v_\z) + \i*(base_1) + \i*(base_2) + 0.5*(base_1) + 0.5*(base_2)$)
					--
					($(v_\z) + \i*(base_1) + \i*(base_2) - 0.5*(base_1) + 0.5*(base_2)$)
					--
					cycle;
			}
		}
		\foreach \z in {7}{
			\coordinate (v_\z) at ($\z*(base_2) + 0.5*(base_2) + 2*\z*(base_1) + 0.5*(base_1)$);
			\fill[draw=black, fill=white, opacity=0.5]
				($(v_\z) + 0*(base_1) - 0.5*(base_1) - 0.5*(base_2)$)
				--
				($(v_\z) + 0*(base_1) + 0.5*(base_1) - 0.5*(base_2)$)
				--
				($(v_\z) + 0*(base_1) + 0.5*(base_1) + 0.5*(base_2)$)
				--
				($(v_\z) + 0*(base_1) - 0.5*(base_1) + 0.5*(base_2)$)
				--
				cycle;
		}
		\node at (v_0){$1$};
		\foreach \i in {1}{
			\foreach \j in {0,\i}{
				\node at ($(v_0) + \i*(base_1) + \j*(base_2)$){$2$};
			}
		}
		\foreach \i in {2}{
			\foreach \j in {0,\i}{
				\node at ($(v_0) + \i*(base_1) + \j*(base_2)$){$4$};
			}
			\node at ($(v_0) + \i*(base_1) + 1*(base_2)$){$3$};
		}
		\foreach \i in {3}{
			\foreach \j in {0,\i}{
				\node at ($(v_0) + \i*(base_1) + \j*(base_2)$){$3$};
			}
			\foreach \j in {1,\i-1}{
				\node at ($(v_0) + \i*(base_1) + \j*(base_2)$){$7$};
			}
		}
		\foreach \i in {4}{
			\foreach \j in {0,2,\i}{
				\node at ($(v_0) + \i*(base_1) + \j*(base_2)$){$6$};
			}
			\foreach \j in {1,\i-1}{
				\node at ($(v_0) + \i*(base_1) + \j*(base_2)$){$5$};
			}
		}
		\foreach \i in {5}{
			\foreach \j in {0,\i}{
				\node at ($(v_0) + \i*(base_1) + \j*(base_2)$){$7$};
			}
			\foreach \j in {1,2,3,4}{
				\node at ($(v_0) + \i*(base_1) + \j*(base_2)$){$8$};
			}
		}
		\foreach \i in {6}{
			\foreach \j in {0,\i}{
				\node at ($(v_0) + \i*(base_1) + \j*(base_2)$){$8$};
			}
			\foreach \j in {1,\i-1}{
				\node at ($(v_0) + \i*(base_1) + \j*(base_2)$){$16$};
			}
			\foreach \j in {2,\i-2}{
				\node at ($(v_0) + \i*(base_1) + \j*(base_2)$){$10$};
			}
			\node at ($(v_0) + \i*(base_1) + 3*(base_2)$){$9$};
		}
		\foreach \i in {7}{
			\foreach \j in {0,1,\i-1,\i}{
				\node at ($(v_0) + \i*(base_1) + \j*(base_2)$){$11$};
			}
			\foreach \j in {2,\i-2}{
				\node at ($(v_0) + \i*(base_1) + \j*(base_2)$){$9$};
			}
			\foreach \j in {3,\i-3}{
				\node at ($(v_0) + \i*(base_1) + \j*(base_2)$){$12$};
			}
		}
		\foreach \i in {8}{
			\foreach \j in {0,\i}{
				\node at ($(v_0) + \i*(base_1) + \j*(base_2)$){$16$};
			}
			\foreach \j in {1,\i-1}{
				\node at ($(v_0) + \i*(base_1) + \j*(base_2)$){$9$};
			}
			\foreach \j in {2,\i-2}{
				\node at ($(v_0) + \i*(base_1) + \j*(base_2)$){$15$};
			}
			\foreach \j in {3,\i-3}{
				\node at ($(v_0) + \i*(base_1) + \j*(base_2)$){$20$};
			}
			\node at ($(v_0) + \i*(base_1) + 4*(base_2)$){$16$};
		}
		\foreach \i in {9}{
			\foreach \j in {0,\i}{
				\node at ($(v_0) + \i*(base_1) + \j*(base_2)$){$24$};
			}
			\foreach \j in {1,\i-1}{
				\node at ($(v_0) + \i*(base_1) + \j*(base_2)$){$23$};
			}
			\foreach \j in {2,\i-2}{
				\node at ($(v_0) + \i*(base_1) + \j*(base_2)$){$17$};
			}
			\foreach \j in {3,\i-3}{
				\node at ($(v_0) + \i*(base_1) + \j*(base_2)$){$19$};
			}
			\foreach \j in {4,\i-4}{
				\node at ($(v_0) + \i*(base_1) + \j*(base_2)$){$26$};
			}
		}
		\foreach \i in {10}{
			\foreach \j in {0,\i}{
				\node at ($(v_0) + \i*(base_1) + \j*(base_2)$){$5$};
			}
			\foreach \j in {1,3,\i-3,\i-1}{
				\node at ($(v_0) + \i*(base_1) + \j*(base_2)$){$18$};
			}
			\foreach \j in {2,\i-2}{
				\node at ($(v_0) + \i*(base_1) + \j*(base_2)$){$26$};
			}
			\foreach \j in {4,\i-4}{
				\node at ($(v_0) + \i*(base_1) + \j*(base_2)$){$24$};
			}
			\node at ($(v_0) + \i*(base_1) + 5*(base_2)$){$22$};
		}
		\foreach \i in {11}{
			\foreach \j in {0,\i}{
				\node at ($(v_0) + \i*(base_1) + \j*(base_2)$){$17$};
			}
			\foreach \j in {1,\i-1}{
				\node at ($(v_0) + \i*(base_1) + \j*(base_2)$){$25$};
			}
			\foreach \j in {2,4,5,\i-5,\i-4,\i-2}{
				\node at ($(v_0) + \i*(base_1) + \j*(base_2)$){$32$};
			}
			\foreach \j in {3,\i-3}{
				\node at ($(v_0) + \i*(base_1) + \j*(base_2)$){$22$};
			}
		}
		\foreach \i in {12}{
			\foreach \j in {0,\i}{
				\node at ($(v_0) + \i*(base_1) + \j*(base_2)$){$13$};
			}
			\foreach \j in {1,\i-1}{
				\node at ($(v_0) + \i*(base_1) + \j*(base_2)$){$10$};
			}
			\foreach \j in {2,\i-2}{
				\node at ($(v_0) + \i*(base_1) + \j*(base_2)$){$34$};
			}
			\foreach \j in {3,\i-3}{
				\node at ($(v_0) + \i*(base_1) + \j*(base_2)$){$29$};
			}
			\foreach \j in {4,\i-4}{
				\node at ($(v_0) + \i*(base_1) + \j*(base_2)$){$35$};
			}
			\foreach \j in {5,\i-5}{
				\node at ($(v_0) + \i*(base_1) + \j*(base_2)$){$34$};
			}
			\node at ($(v_0) + \i*(base_1) + 6*(base_2)$){$17$};
		}
		\foreach \i in {13}{
			\foreach \j in {0,\i}{
				\node at ($(v_0) + \i*(base_1) + \j*(base_2)$){$12$};
			}
			\foreach \j in {1,4,\i-4,\i-1}{
				\node at ($(v_0) + \i*(base_1) + \j*(base_2)$){$21$};
			}
			\foreach \j in {2,\i-2}{
				\node at ($(v_0) + \i*(base_1) + \j*(base_2)$){$20$};
			}
			\foreach \j in {3,\i-3}{
				\node at ($(v_0) + \i*(base_1) + \j*(base_2)$){$14$};
			}
			\foreach \j in {5,\i-5}{
				\node at ($(v_0) + \i*(base_1) + \j*(base_2)$){$38$};
			}
			\foreach \j in {6,\i-6}{
				\node at ($(v_0) + \i*(base_1) + \j*(base_2)$){$13$};
			}
		}
		\foreach \i in {14}{
			\foreach \j in {0,1,\i-1,\i}{
				\node at ($(v_0) + \i*(base_1) + \j*(base_2)$){$32$};
			}
			\foreach \j in {2,\i-2}{
				\node at ($(v_0) + \i*(base_1) + \j*(base_2)$){$14$};
			}
			\foreach \j in {3,\i-3}{
				\node at ($(v_0) + \i*(base_1) + \j*(base_2)$){$42$};
			}
			\foreach \j in {4,\i-4}{
				\node at ($(v_0) + \i*(base_1) + \j*(base_2)$){$38$};
			}
			\foreach \j in {5,\i-5}{
				\node at ($(v_0) + \i*(base_1) + \j*(base_2)$){$43$};
			}
			\foreach \j in {6,\i-6}{
				\node at ($(v_0) + \i*(base_1) + \j*(base_2)$){$33$};
			}
			\node at ($(v_0) + \i*(base_1) + 7*(base_2)$){$33$};
		}
	\end{tikzpicture}
	\end{center}
	\caption{
		Sprague-Grundy values
		of the ruler
		$\mathscr{G}\left( \ASM{16}, \mathscr{T} \right)$
	}
	\label{table:G_value-ruler}
\end{table}
\par
%
%
%
%
%
%
\subsection{The Ruler on the Poset of Set Partitions}\label{ssc:OpenSetPar}
In this subsection we consider the Grundy functions of the coin turning games on
the poset $X=\Pi_{n}$ of set partitions.
The order ideal game on $X$ is already settled 
by Corollary~\ref{cor:OrderI},
and therefore we focus on the ruler on $X$.
We show that the Grundy value of $\pi\in\Pi_{n}$
depends only on its type (defined below),
which is the integer partition given by the block sizes of $\pi$.
Moreover, we show that it suffices to determine the Grundy value for the one-line partition $(n)$.
However, 
we are unable to find a reasonable conjectural formula for the Grundy value in this case, 
which we denote by $h(n)$,
(see Table~\ref{tb:h(n)}).
\par\smallskip
An \textsl{integer partition} (or simply a \textsl{partition}) is a finite sequence
\begin{equation}\label{eq:par}
\lambda=(\lambda_{1},\lambda_{2},\dots,\lambda_{r})
\end{equation}
of non-negative integers in weakly decreasing order:
$\lambda_{1}\geq\lambda_{2}\geq\cdots\geq\lambda_{r}$.
The nonzero entries $\lambda_{i}$, in \eqref{eq:par} are called the \textsl{parts} of $\lambda$.
The number of parts 
is the length of $\lambda$, denoted by $\ell(\lambda)$, 
and the sum of the parts is the weight
of $\lambda$, denoted by $|\lambda|$.
If $|\lambda|=n$, we say that $\lambda$ is a partition of $n$.
It is sometimes convenient to use the notation
$\lambda=(i^{m_{i}(\lambda)})_{i\in[n]}=(1^{m_{1}(\lambda)}2^{m_{2}(\lambda)}\dots)$,
which indicates that exactly $m_{i}(\lambda)$ of the parts of $\lambda$ are equal to $i$.
The integer $m_{i}(\lambda)$ (or simply $m_{i}$) is called the \textsl{multiplicity} of $i$ in $\lambda$.
The set of all partitions of $n$ is
denoted by $\Par_{n}$.
%
If $\lambda=(i^{m_{i}(\lambda)})_{i\in[n]}$ and $\mu=(i^{m_{i}(\mu)})_{i\in[n]}$ are partitions, 
we define their union $\lambda\cup\mu$ to be the partition whose multiplicities are given by 
$m_{i}(\lambda\cup\mu)=m_{i}(\lambda)+m_{i}(\mu)$.
For example,
$\lambda=(2,1,1)=(1^22)$ is a partition of $4$ with length $3$,
and $\Par_{4}=\{4,13,2^2,1^22,1^4\}$.
Further if $\mu=(1^32^23)$, then $\lambda\cup\mu=(1^52^33)$.
\begin{figure}[hbt]
\adjustbox{scale=0.8,center}{%
\begin{tikzcd}
                           & 4  &      \\
 13   \arrow{ru}  &   &   2^2 \arrow{lu}                       \\
   & 1^22  \arrow{lu}\arrow{ru}  &  \\
    & 1^4 \arrow{u}   & \\
\end{tikzcd}
}
\caption{$\Par_{4}$\label{fig:Par4}}
\end{figure}
\par
Let $\lambda$ and $\mu$ be partitions of $n$. 
We say that $\mu$ \textsl{refines} $\lambda$,
denoted by $\mu\leq\lambda$, 
if each part of $\lambda$ can be written as a sum of some parts of $\mu$,
with every part of $\mu$ used exactly once.
Equivalently, writing $\lambda=(\lambda_{1},\dots,\lambda_{r})$, 
there exists an $r$-tuple $\bnu=(\nu^{(1)},\dots,\nu^{(r)})$ of partitions such that $|\nu^{(i)}|=\lambda_{i}$ for each $i$, 
and $\mu=\nu^{(1)}\cup\cdots\cup\nu^{(r)}$.
This relation defines the \textsl{refinement order} on $\Par_{n}$.
For example, $(5,5,1)$ refines $(6,5)$, but does not refine $(7,4)$.
The Hasse diagram of the poset $\Par_{4}$ is depicted in Figure~\ref{fig:Par4}.
\par\smallskip
Here we regard the tuple $\bnu$ as a multiset of partitions, 
ignoring the order of the components; 
that is, tuples such as $(1^2,2,2)$ and $(2,1^2,2)$ are identified.
Note that the choice of $\bnu$ is not unique.
Let $\Par_{n}(\lambda,\mu)$ denote the set of all such multisets $\bnu$ of partitions.
For example, 
$\Par_{n}(\lambda,\mu)=\{(2^2,1^2),(1^22,2)\}$ for $\lambda=(4,2)$ and $\mu=(2,2,1,1)=(1^22^2)$.
\par\smallskip
For each $\pi=\{\pi_{1},\pi_{2},\dots,\pi_{k}\}\in\Pi_{n}$, 
we associate the integer partition $\lambda=(|\pi_{1}|,|\pi_{2}|,\dots,|\pi_{k}|)\in\Par_{n}$,
 arranged in weakly decreasing order.
Here $|B|$ denotes the number of elements in the block $B$.
This map $\fnb:\Pi_{n}\to\Par_{n},\,\pi \mapsto \lambda$ is order-preserving.
The integer partition $\lambda=\fnb(\pi)$ is called the \textsl{type} of the set partition $\pi$.
For $\pi\in\Pi_{n}$, we henceforth write $g_{n}(\pi)$ for the Grundy value $g_{\Pi_{n},\I}(\pi)$.
%
%
%
\begin{thm}\label{th:SetParGrundy}
Let $X=\Pi_{n}$ be the poset of set partitions of $[n]$.
We consider the ruler $\Pgame(X,\I)$.
\begin{enumerate}[label=(\roman*)]
\item\label{it:Pi-1}
For $\pi\in X$, the Grundy value $g_{n}(\pi)$ depends only on its type $\lambda=\fnb(\pi)$.
Accordingly, we may write $g_{n}(\lambda)$ for $g_{n}(\pi)$.
\item\label{it:Pi-2}
If $\lambda=(\lambda_{1},\lambda_{2},\dots,\lambda_{r})$ then
$g_{n}(\lambda)=g_{\lambda_{1}}(\lambda_{1})\nimmul g_{\lambda_{2}}(\lambda_{2})\nimmul\cdots \nimmul g_{\lambda_{r}}(\lambda_{r})$.
\end{enumerate}
\end{thm}
\begin{demo}{Proof}
Note that $X$ has the minimum element $\hat0=\{1,2,\dots,n\}$.
Let $\pi=\{\pi_{1},\pi_{2},\dots\}\in X$
be a set partition of $[n]$ of type $\lambda$.
We may label the blocks so that $\pi_{1}=\{i_{1},\dots,i_{\lambda_{1}}\}$,
$\pi_{2}=\{i_{\lambda_{1}+1},\dots,i_{\lambda_{1}+\lambda_{2}}\},\dots$.
\begin{enumerate}[label=(\roman*)]
\item
Let $\sigma=\{\sigma_{1},\sigma_{2},\dots\}\in X$ be another set partition of $[n]$ of the same type $\lambda$.
We label the blocks of $\sigma$ so that 
$\sigma_{1}=\{j_{1},\dots,j_{\lambda_{1}}\}$,
$\sigma_{2}=\{j_{\lambda_{1}+1},\dots,j_{\lambda_{1}+\lambda_{2}}\}\dots$.
Define a bijection $f:[n]\to[n]$ by $f(i_{1})=j_{1},\dots,f(i_{\lambda_{1}})=j_{\lambda_{1}},\dots$.
Then $f$ induces an order automorphism of $\Pi_{n}$ such that $[\hat0,\pi]\simeq[\hat0,\sigma]$ 
and this automorphism satisfies condition~\eqref{eq:iso-T} with $\T_{1}=\T_{2}=\I$.
Hence, by Proposition~\ref{prop:g_value_order_isomorphism_general},
we obtain $g_{n}(\pi)=g_{n}(\sigma)$.
\item
Since $[\hat0,\pi]\simeq\Pi_{\lambda_{1}}(\{\{i_{1},\dots,i_{\lambda_{1}}\})\times\Pi_{\lambda_{2}}(\{\{i_{\lambda_{1}+1},\dots,i_{\lambda_{1}+\lambda_{2}}\})\times\cdots$,
Theorem~\ref{th:nim-prod-app} implies our claim.
\end{enumerate}
\end{demo}
By Theorem~\ref{th:SetParGrundy},
it suffices to compute $g_{n}((n))$ for each nonnegative integer $n$ 
in order to determine the Grundy value $g_{n}(\pi)$ for any $\pi\in\Pi_{n}$.
Henceforth, we write $h(n)$ for $g_{n}((n))$.
Table~\ref{tb:h(n)} lists the first few values of $h(n)$.
At present, we are not aware of any simple explicit formula or conjectural pattern for $h(n)$.
\begin{table}[h]
 \begin{center}
	\begin{tabular}{|l||c|c|c|c|c|c|c|c|c|c|c|c|c|c|c|c|c|c|c|c|c|} 
	\hline
$n$     &   1 &  2 &  3 &  4 &  5 &  6 &  7 &  8 &  9 & 10 & 11 & 12 & 13 & 14 & 15 & 16 & 17 
\\\hline
$h(n)$ &   1 &  2 &  1 &  4 &  1 &  2 &  1 &  7 & 15 & 16 &  8 &  5 & 19 &  5 & 37 & 17 & 14 
\\\hline
	\end{tabular}
   \caption{Grundy values $h(n)$}\label{tb:h(n)}
 \end{center}
\end{table}
\par
%
%
%
\begin{prop}\label{lem:M(lambda,mu)}
Let $\lambda=(i^{m_{i}(\lambda)})_{i\in[n]}$
and $\mu=(i^{m_{i}(\mu)})_{i\in[n]}$
be partitions of $n$ such that $\lambda\geq\mu$.
%
Let $\{\alpha^{(1)},\dots,\alpha^{(s)}\}$ be the set of all distinct partitions that appear as components of some $\bnu\in\Par_{n}(\lambda,\mu)$.
Let $m_{\alpha^{(k)}}(\bnu)$ denote the multiplicity of $\alpha^{(k)}$ in $\bnu\in\Par_{n}(\lambda,\mu)$.
If $\pi\in\Pi_{n}$ be a set partition of $[n]$ of type $\mu$,
then the number of set partitions $\tau\in\Pi_{n}$ of type $\lambda$ satisfying $\tau\geq\pi$
 is given by
\begin{equation}\label{eq:M(lambda,mu)}
M_{n}(\lambda,\mu)=\sum_{\nu\in\Par_{n}(\lambda,\mu)}
\frac{\prod_{i=1}^{n}m_{i}(\mu)!}
{\prod_{j=1}^{\ell(\lambda)}\prod_{i=1}^{n}m_{i}(\nu^{(j)})!\prod_{k=1}^{s}m_{\alpha^{(k)}}(\bnu)!}.
\end{equation}
\end{prop}
The proof of this proposition is elementary and is left to the reader.
At the end of this subsection, 
we explain how to compute $h(n)$ inductively using Theorem~\ref{th:SetParGrundy} and Table~\ref{tb:h(n)}.
By Theorem~\ref{thm:FT_of_T_game}, we have
$h(n)=g_{n}(\{[n]\})=\mex\{s_{n}(\pi)\mid\pi\in\Pi_{n}\}$,
where
\begin{equation}\label{eq:sn(pi)}
s_{n}(\pi):=\underset{{\tau\in\Pi_{n}}\atop{\tau\geq\pi}}{\,\,\nimsum}g_{n}(\tau).
\end{equation}
Note that the NIM-sum $s_{n}(\pi)$
depends only on the type $\mu=\fnb(\pi)\in\Par_{n}$;
therefore, we denote this value by $s_{n}(\mu)$.
By Theorem~\ref{th:SetParGrundy} \ref{it:Pi-1}, 
the Grundy value $g_{n}(\tau)$ depends only on the type $\lambda=\fnb(\tau)$.
Moreover, by Theorem~\ref{th:SetParGrundy} \ref{it:Pi-2}, 
this value $g_{n}(\lambda)$ can be computed inductively as the NIM-product of its components.
Furthermore, the number of set partitions $\tau$ of type $\lambda$ such that $\tau \geq \pi$ 
is given by $M_{n}(\lambda,\mu)$
(see Proposition~\ref{lem:M(lambda,mu)} \eqref{eq:M(lambda,mu)}).
%
Hence, \eqref{eq:sn(pi)} can be written as
\begin{equation}\label{eq:sn(mu)}
s_{n}(\mu)=\underset{{\lambda\in\Par_{n}}\atop{\lambda\geq\mu}}{\,\,\nimsum}M_{n}(\lambda,\mu) g_{n}(\lambda).
\end{equation}
We then use the formula $h(n)=\mex\{s_{n}(\mu)|\mu\in\Par_{n}\}$ to determine $h(n)$.
\par
For example, we illustrate the case $n=4$.
The poset $\Par_{4}$ of types is depicted in Figure~\ref{fig:Par4}.
The Grundy values $g_{4}(\lambda)$,
computed via the NIM-product for each $\lambda\in\Par_{4}\setminus\{4\}$, 
are shown in Figure~\ref{fig:Grundy-Pi4},
and the computed values of $s_{4}(\mu)$ are listed in Figure~\ref{fig:s6(mu)}.
Taking the minimum excluded value of these numbers, we obtain
$\mex\{s_{4}(\mu)\mid \mu\in\Par_{4}\}=4$, and hence $h(4)=4$.
Below, we present an explicit example illustrating how the values $s_{4}(\mu)$ are computed.
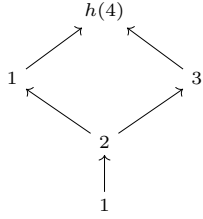
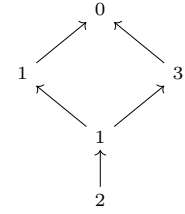
\begin{figure}[htbp]
{\tiny
\centering
\begin{subfigure}[b]{0.4\textwidth}
    \centering
\begin{tikzcd}
           &    h(4)  &   \\
 1   \arrow{ru}    &   &   3 \arrow{lu}   \\
   &  2 \arrow{lu}\arrow{ru} &     \\
  & 1 \arrow{u} & \\
\end{tikzcd}
    \caption{
The value of $g_{4}(\lambda)$}
    \label{fig:Grundy-Pi4}
\end{subfigure}
}
\hfill
{\tiny
\begin{subfigure}[b]{0.4\textwidth}
    \centering
\begin{tikzcd}
   &  0       & \\
 1   \arrow{ru}  &         &   3 \arrow{lu}       \\
   & 1 \arrow{lu}\arrow{ru}     &  \\
   & 2 \arrow{u} &  \\
\end{tikzcd}
    \caption{The value of $s_{4}(\mu)$}
    \label{fig:s6(mu)}
\end{subfigure}
}
\caption{The values of $g_{4}(\lambda)$ and $s_{4}(\mu)$ for $\Par_{4}$}
\label{fig:s6Par}
\end{figure}
Let $I_{\mu}$ denote the interval $[\mu,6]$ in $\Par_{6}$ for $\mu\in\Par_{6}$.
First observe that 
if $\mu=(4)$, then the nim-sum defining $s_{4}(4)$ is empty, and hence $s_{4}(4)=0$.
As an illustrative example, take $\mu=(1^22)$.
Then we have $I_{\mu}\setminus\{(4)\}=\{1^22,13,2^2\}$.
For each $\lambda\in I_{\mu}\setminus\{(4)\}$, we compute $\Par_{4}(\lambda,\mu)$ 
and the corresponding multiplicity $M(\lambda,\mu)$.
\begin{enumerate}[label=\alph*)]
\item
If $\lambda=1^22$, then $\Par_{4}(\lambda,\mu)=\{(1,1,2)\}$ and $M(\lambda,\mu)=\frac{2!1!}{1!1!1!2!1!}=1$.
\item
If $\lambda=13$, then $\Par_{6}(\lambda,\mu)=\{(1,12)\}$ and $M(\lambda,\mu)=\frac{2!1!}{1!1!1!1!1!}=2$.
\item
If $\lambda=2^2$, then $\Par_{6}(\lambda,\mu)=\{(1^2,2)\}$ and $M(\lambda,\mu)=\frac{2!1!}{2!1!1!1!}=1$.
\end{enumerate}
Consequently, we obtain
\[
s_{4}(1^22)=1\cdot2\nimadd 2\cdot1\nimadd 1\cdot3 =1.
\]
The remaining values of $s_{4}(\mu)$ in Figure~\ref{fig:s6(mu)} are computed in the same manner.
Hence, we conclude that $h(4)=4$.
\par
Of course, the Grundy value $g_{4}(4)$ can be computed directly from the Hasse diagram of $\Pi_{4}$
(see Figure~\ref{fig:Pi4}).
However, as $n$ increases, constructing the increasingly complicated Hasse diagram of $\Pi_{n}$ becomes time-consuming.
%
%
%
%

%
%
%
%
\appendix
%
\section{NIM-addition and NIM-multiplication}\label{sc:appendix}
First, we recall elementary facts on the minimal-excluded number in the following proposition.
%
\begin{prop}\label{prop:NIM_subset}
	Let $S$ and $T$ be proper subsets of $\Non$
	such that $S$ is a subset of $T$
	(i.e., $S \subset T \subsetneq \Non$).
	Then the followings hold:
	\begin{enumerate}[label=(\roman*)] 
		\item
			$\mex(S) \leq \mex(T)$,
		\item
			$\mex(S) = \mex(T)$
			if $\mex(S) \notin T$.
	\end{enumerate} 
\end{prop}
This proposition helps to prove the following proposition for the NIM-addition.
\begin{prop}\label{prop:FT_NIM_add}
	Let $a$, $b$ and $c$ be non-negative integers.
	The NIM-addition has the following properties:
	\begin{enumerate}[label=(\Roman*)] 
		\item\label{property:injectivity_of_NIM_add}
			$a \nimadd b \neq a \nimadd c$ if $b \neq c$,
		\item	commutative (i.e., $a \nimadd b = b \nimadd a$),
		\item	$a \nimadd 0 = 0 \nimadd a = a$,
		\item	$a \nimadd a = 0$,
		\item	associative (i.e., $(a \nimadd b) \nimadd c = a \nimadd (b \nimadd c)$).
	\end{enumerate} 
\end{prop}
\begin{demo}{Proof}
%
The proof is straightforward and left to the reader.
\end{demo}
%
%
\begin{lem}
For any nonnengative integers $a,b,c\in\Non$ such that $c<a\oplus b$,
there exists $a'<a$ with $a'\oplus b=c$ or $b'<b$ with $a\oplus b'=c$.
\end{lem}
\begin{demo}{Proof}
We claim that, for any integer $c\in\Non$ such that $0\leq c<a\oplus b$,
exactly one of $b\oplus c< a$ or $a\oplus c < b$ holds.
Then we may put $a'=b\oplus c$ or $b'=a\oplus c$ to obtain $a'$ or $b'$ satisfying this lemma.
%
%
%
\par\smallskip
We follow the notation in Definition~\ref{def:binary}.
We don't need the second case here.
\begin{enumerate}[label=(\roman*)] 
\item
Assume $b\oplus c>a$.
If we set $i_{0}:=\max\{i\,|\,\digit{i}{b\oplus c}\neq\digit{i}{a}\}$,
then we must have $\digit{i_{0}}{b\oplus c}=1$ and $\digit{i_{0}}{a}=0$.
\par
If $\digit{i_{0}}{a\oplus b}=0$ could hold,
then we would have $\digit{i_{0}}{b}=\digit{i_{0}}{a}=0$ and $\digit{i_{0}}{c}=\digit{i_{0}}{b\oplus(b\oplus c)}=1$.
From the assumption we also would have $\digit{i}{b\oplus c}=\digit{i}{a}$ for $i>i_{0}$,
which implies $\digit{i}{a\oplus b}=\digit{i}{(b\oplus c)\oplus b}=\digit{i}{c}$.
Further $\digit{i_{0}}{a\oplus b}=0$ and $\digit{i_{0}}{c}=1$ implies $a\oplus b<c$ which is a contradiction.
Hence we coclude that $\digit{i_{0}}{a\oplus b}=1$.
\par
Thus, in this case, we have 
$\digit{i}{a\oplus c}=\digit{i}{(a\oplus b)\oplus(b\oplus c)}=\digit{i}{(a\oplus b)\oplus a}=\digit{i}{b}$
for $i>i_{0}$,
and, further, we have
$\digit{i_{0}}{a\oplus c}=\digit{i_{0}}{(a\oplus b)\oplus(b\oplus c)}=0$ and 
$\digit{i_{0}}{b}=\digit{i_{0}}{(a\oplus b)\oplus a}=1$.
Hence we coclude that $a\oplus c<b$.
\item
In the case 
$b\oplus c<a$, we can show that $a\oplus c>b$ similarly.
\end{enumerate} 
This completes the proof.
\end{demo}
%
%
%
\begin{thm}\label{thm:nimaddd=oplus}
If $a,b\in\Non$, then we have
\[
a\nimadd b = a \oplus b.
\]
\end{thm}
\begin{demo}{Proof}
For $a'\neq a$, $b'\neq b$ we certainly have $a'\oplus b\neq a\oplus b \neq a\oplus b'$.
We proceed by induction on $a+b$.
If $a+b=0$ then we have $a=b=0$,
hence $0\nimadd 0=0$ and $0\oplus 0$ implies the both sides coincide.
For $a+b>0$, by induction hyperthesis, we have
\begin{align*}
a\nimadd b
=\mex(\{ a'\nimadd b,\, a\nimadd b' \,|\, a'<a,\, b'<b \} )
=\mex\left(\left\{a'\oplus b,\, a\oplus b' \,\middle|\, a'<a,\, b'<b\right\} \right).
\end{align*}
If we set $S:=\left\{a'\oplus b,\, a\oplus b' \,\middle|\, a'<a,\, b'<b\right\}$,
then $\{c\in\Non \,|\, c<a\oplus b\}\subseteq S$ and $a\oplus b\not\in S$ shows $\mex(S)=a\oplus b$,
Hence the corollary follows.
\end{demo}
%
\begin{lem}\label{lemma:NIM_add_and_mex}
	Let $S$ and $T$ be proper subsets of $\Non$ such that $\mex(S)=a$ and $\mex(T)=b$.
	Then we have the following equation:
	\begin{equation*}
		a \nimadd b
		=
		\mex(
			\{ s \nimadd b \, | \, s \in S \}
			\cup
			\{ a \nimadd t \, | \, t \in T \}
		).
	\end{equation*}
\end{lem}
\begin{demo}{Proof}
	First,
	the inequality
	$(\textrm{LHS}) \leq (\textrm{RHS})$
	follows from proposition~\ref{prop:NIM_subset}
	because
	a non-negative $x$ belongs to $S$ (resp. $T$)
	if $x$ is less than $\mex(S)$ (resp. $\mex(T)$).
	From the property~\eqref{property:injectivity_of_NIM_add} 
	of Proposition~\ref{prop:FT_NIM_add},
	the NIM-sum
	$\mex(S) \nimadd \mex(T)$ belongs to
	neither
	$\{ a \nimadd \mex(T) \, | \, a \in S \}$
	nor
	$\{ \mex(S) \nimadd b \, | \, b \in T \}$.
	Therefore
	the equation
	$(\textrm{LHS}) = (\textrm{RHS})$
	follows.
\end{demo}
%
The following corollary is easy to prove by induction.
\begin{cor}\label{cor:NIM_add_and_mex}
	Let $n$ be a positive integer, and
	$S_{1}, S_{2}, \ldots , S_{n}$ proper subsets of $\Non$
	such that $\mex S_{i}=a_{i}$ for $1\leq i\leq n$.
	Then we have the following equation:
	\begin{equation}\label{eq:NIM_add_and_mex_multi}
		\overset{n}{\underset{i=1}{\nimsum}} a_{i}
		=
		\mex\Bigl(
			\bigcup\limits_{i=1}^{n}
			\Bigl\{
				\underset{
					\substack{
						1 \leq k \leq n	\\
						k \neq i
					}
				}{
				\nimsum
				}
				a_{k}
				\nimadd s
			\, \Big| \,
				s \in {S}_{i}
			\Bigr\}
		\Bigr)
		.
	\end{equation}
\end{cor}
%
%
%
%
\begin{prop}
Let $a,b,c,a_{1},a_{2}\in\Non$.
Then we have
\begin{enumerate}[label=(\roman*)] 
\item
$a\nimmul 0=0$, $a\nimmul 1=a$,
\item
$a \nimmul b=b \nimmul a$,
\item
If $b\neq0$, then $a_{1} \nimmul b=a_{2} \nimmul b$ if and only if $a_{1}~a_{2}$,
\item
$(a \nimadd b) \nimmul c = a \nimmul c \nimadd b \nimmul c$,
\item
$(a \nimmul b) \nimmul c = a \nimmul (b \nimmul c)$.
\end{enumerate} 
\end{prop}
%
%
%
%
%
%
\section{A Characterization of the Ruler Sequence}\label{sc:RC}
The purpose of this appendix is to prove Lemma~\ref{lem:chain-ruler-recurrence},
which characterize the ruler sequence by means of minimum excluded value.
For this purpose, we define 
\begin{equation}
\fnH{m}{n}=\overset{n-1}{\underset{x=m}{\,\,\,\nimsum}}\fnh{x}
\end{equation}
for positive integers $m$ and $n$ such that $m\leq n$,
where $\fnh{x}$ is the ruler sequence defined in Definition~\ref{def:binary}.
Note that $\fnH{n}{n}=0$ since empty NIM-sum is equal to $0$.
For example, $\fnH{3}{6}=\fnh{3}\oplus\fnh{4}\oplus\fnh{5}=1\oplus4\oplus1=4$.
Further, for a nonnegative integer $n$,
let $\fnS{n}:=\left\{\fnH{x}{n} \,\middle|\, x=1,\dots,n\right\}$.
For instance,
$\fnS{6}=\left\{ \fnH{1}{6},\fnH{2}{6},\fnH{3}{6},\fnH{4}{6},\fnH{5}{6},\fnH{6}{6} \right\}
=\left\{ 0,1,4,5,6,7\right\}$.
%
\begin{lem}\label{lem:nim-nonzero}
We have
\[
\fnH{m}{n}\neq0
\]
for all integers $m,n$ such that $1\leq m<n$.
\end{lem}
\begin{demo}{Proof}
Let $r_{0}$ be the largest integer such that the interval length $n-m$ satisfy $2^{r_{0}}\leq n-m$.
Let $r_{1}:=\max\{\placemin{x} \,|\, m\leq x\leq n-1\}$.
Then we have $r_{1}\geq r_{0}$ since there exists an integer divisible by $2^{r_{0}}$ 
among $2^{r_{0}}$ consecutive integers.
\par
If $x\in[m, n-1]$ such that $\placemin{x}=r_{1}$ is unique,
then $\fnH{m}{n}\neq0$ 
since the binary digit in position $r_{1}$ does not cancel and therefore remains in $\fnH{m}{n}$.
\par\smallskip
Next, we show that there cannot be two distinct integers $x\in[m, n-1]$ such that $\placemin{x}=r_{1}$.
Suppose, for the sake of contradiction, that there are two such integers.
Then necessarily $r_{1}=r_{0}$,
and there can be at most two integers $x$ with $\placemin{x}=r_{0}$ because of the length of the interval.
Hence there would exist $m\leq x_{0}<x_{1}\leq n-1$ such that $x_{1}=x_{0}+2^{r_{0}}$ and $\placemin{x_{0}}=\placemin{x_{1}}=r_{0}$.
By assumption, there exists an odd integer $y_{0}\in\Non$ such that $x_{0}=2^{r_{0}}y_{0}$,
which implies $x_{1}=2^{r_{0}}(y_{0}+1)$.
Since $y_{0}+1$ is even, we obtain $\placemin{x_{1}}>r_{0}$,
contradicting the assumption $\placemin{x}=r_{0}$.
This contradiction completes the proof.
%
%
\end{demo}
%
%
%
\begin{lem}\label{lem:chain-ruler}
We have the followings.
\begin{enumerate}[label=\roman*)]
\item\label{it:chain1}
For fixed $n$, the numbers $\fnH{m}{n}$ ($m=1,\dots,n$) are all distinct.
\item\label{it:chain2}
$\fnS{2^{k}}=\{0,1,\cdots,2^{k}-1\}$ for any $k\in\Non$.
\item\label{it:chain3}
If $k=\placemin{n}$,
then $\fnS{2^{k}}\subseteq \fnS{n}$.
\item\label{it:chain4}
If $k=\placemin{n}$,
then $2^{k}\not\in \fnS{n}$.
\end{enumerate}
\end{lem}
\begin{demo}{Proof}
\begin{enumerate}[label=\roman*)]
\item
If $1\leq m_{1}< m_{2}\leq n$,
then we have
\[
\fnH{m_{1}}{m_{2}}\oplus\fnH{m_{2}}{n}=\fnH{m_{1}}{n}.
\]
Since Lemma~\ref{lem:nim-nonzero} implies $\fnH{m_{1}}{m_{2}}\neq0$,
we obtain $\fnH{m_{1}}{n}\neq\fnH{m_{2}}{n}$.
\item
If $1\leq x<2^{k}$, then $\placemin{x}\leq\placemax{x}<k$.
Consequently, each term $\fnh{x}=2^{\placemin{x}}$ appearing in the NIM-sum of $\fnH{m}{2^{k}}$ is strictly less than $2^k$, and therefore
$0\leq \fnH{m}{2^{k}}<2^{k}$ for any $1\leq m\leq 2^{k}$.
Since the values $\fnH{m}{2^{k}}$ ($1\leq m\leq 2^{k}$) are pairwise distinct, 
they must form exactly the set $\{0,1,\cdots,2^{k}-1\}$.
\item
Since $k=\placemin{n}$,
the only binary digits in which $n-m$ and $2^{k}-m$
may differ are those in positions strictly larger than $k$ for $m=0,1,\dots,2^{k}-1$.
Hence the position of the least significant nonzero binary digit is the same for both numbers,
which implies that $\fnh{n-m}=\fnh{2^{k}-m}$
for $m=0,1,\dots,2^{k}-1$.
Consequently, we obtain $\fnH{n-m}{n}=\fnH{2^{k}-m}{2^{k}}$ for $m=0,1,\dots,2^{k}-1$.
This proves the claim.
\item
Suppose, for the sake of contradiction, that $2^{k}\in \fnS{n}$.
Then there exists a positive integer $m<n$ such that $\fnH{m}{n}=2^{k}$.
Since $k=\placemin{n}$,
we have $\fnH{n}{n+1}=\fnh{n}=2^{k}$.
Hence,
\[
\fnH{m}{n+1}=\fnH{m}{n}\oplus \fnH{n}{n+1}=2^{k}\oplus 2^{k}=0.
\]
This contradicts Lemma~\ref{lem:nim-nonzero}.
Therefore, we conclude that $2^{k}\not\in \fnS{n}$.
\end{enumerate}
This completes the proof of the lemma.
\end{demo}
Lemma~\ref{lem:chain-ruler}~\ref{it:chain3} and Lemma~\ref{lem:chain-ruler}~\ref{it:chain4} show that $\mex \fnS{n}=2^{k}$.
Since $f(n)=\mex \fnS{n}$,
this completes the proof of Lemma~\ref{lem:chain-ruler-recurrence}.
%
%
%
%
%

%
%
%
%
\addtocontents{toc}{\protect\setcounter{tocdepth}{0}} 
\section*{Acknowledgement}
The authors would like to express their sincere gratitude to Professor Takeshi Suzuki 
for his helpful comments and valuable suggestions.
\addtocontents{toc}{\protect\setcounter{tocdepth}{2}}
%
%
%
%

\end{document}